\documentclass{amsart}
\usepackage[foot]{amsaddr}
\usepackage{amssymb,amsmath,amsfonts,mathptmx,cite,mathrsfs,graphicx,gastex,rotating}
\usepackage{eucal}

\theoremstyle{plain}
\newtheorem{theorem}{Theorem}
\newtheorem{proposition}{Proposition}
\newtheorem{lemma}{Lemma}
\newtheorem{corollary}{Corollary}
\theoremstyle{definition}
\newtheorem{problem}{Problem}
\newtheorem{question}{Question}
\theoremstyle{remark}
\newtheorem{remark}{Remark}

\makeatletter

\renewcommand*\subjclass[2][2010]{\def\@subjclass{#2}\@ifundefined{subjclassname@#1}{\ClassWarning{\@classname}{Unknown edition (#1) of Mathematics Subject Classification; using '2010'.}}{\@xp\let\@xp\subjclassname\csname subjclassname@#1\endcsname}}

\renewcommand{\subjclassname}{\textup{2010} Mathematics Subject Classification}

\makeatother

\begin{document}

\title[Minimal monoids generating varieties with complex subvariety lattices]{Minimal monoids generating varieties with complex subvariety lattices}
\thanks{Supported by the Ministry of Science and Higher Education of the Russian Federation (project FEUZ-2020-0016).}

\author[S. V. Gusev]{Sergey V. Gusev}
\address{Institute of Natural Sciences and Mathematics\\
Ural Federal University\\ 620000 Ekaterinburg, Russia}
\email{sergey.gusb@gmail.com}

\begin{abstract}
A variety is \textit{finitely universal} if its lattice of subvarieties contains an isomorphic copy of every finite lattice.
We show that the 6-element Brandt monoid generates a finitely universal variety of monoids and, by the previous results, it is the smallest generator for a monoid variety with this property.
It is also deduced that the join of two Cross varieties of monoids can be finitely universal.
In particular, we exhibit a finitely universal variety of monoids with uncountably many subvarieties which is the join of two Cross varieties of monoids whose lattices of subvarieties are the 6-element and the 7-element chains, respectively.
\end{abstract}



\maketitle

\section{Introduction}

A \textit{variety} is a class of algebras of a fixed type that is closed under the formation of homomorphic images, subalgebras, and arbitrary direct products.
A variety is \textit{finitely based} if it can be defined by a finite set of identities, otherwise, it is \textit{non-finitely based}.
A variety is \textit{finitely generated} if it is generated by a finite algebra.
A variety is \textit{small} if it contains only finitely many subvarieties.
A finitely generated, finitely based, small variety of algebras is called a \textit{Cross variety}.
Cross varieties have been heavily investigated for many years.
For classical algebras such as groups~\cite{Oates-Powell-64}, associative rings~\cite{Kruse-73,Lvov-73}, and Lie rings~\cite{Bahturin-Olshanskii-75}, every finite member generates a Cross variety.
However, this result is not true for arbitrary algebras.
In general, the variety $\mathbb V$ generated by a finite algebra can be non-Cross in several ways, for instance, $\mathbb V$ can be non-finitely based, the lattice $\mathfrak L(\mathbb V)$ of subvarieties of $\mathbb V$ can be infinite or even uncountable, and $\mathbb V$ can be \textit{finitely universal} in the sense that $\mathfrak L(\mathbb V)$ contains an isomorphic of every finite lattice.

Examples of finitely universal varieties of semigroups have been known since the early 1970s~\cite{Burris-Nelson-71b}, and the smallest semigroup generating such a variety is of order four~\cite{Lee-07}; see Section~12 in the survey~\cite{Shevrin-Vernikov-Volkov-09} for more information.
For a long time, however, it was unknown if finitely universal varieties of monoids exist~\cite[Question 6.3]{Jackson-Lee-18}. 
The first examples of finitely universal varieties of monoids have recently been found~\cite{Gusev-Lee-20}; in fact, there also exist finitely universal varieties that are finitely generated, but an explicit smallest example have not been found; see Section~4 in the very recent survey~\cite{Gusev-Lee-Vernikov-22} for more details.
Unlike semigroups, the variety generated by any monoid of order five or less is not finitely universal~\cite{Gusev-Li-Zhang-23,Lee-Zhang-14}.
This naturally leads to the following problem.

\begin{problem}[{see \cite[Problem 4.7]{Gusev-Lee-Vernikov-22}}]
\label{prob:min-order}
Is there a monoid of order six that generates a finitely universal variety of monoids?
\end{problem}

The 6-element Brandt monoid
\[
B_2^1:=\langle a, b \mid aba = a,\,bab = b,\,aa = bb = 0\rangle\cup\{1\}
\]
is one of the most famous finite monoids.
It can be represented as the matrix semigroup
\[
\begin{tabular}{cccccc}
$\left(\begin{matrix} 0&0\\0&0\end{matrix}\right)$,
&
$\left(\begin{matrix} 1&0\\0&0\end{matrix}\right)$,
&
$\left(\begin{matrix} 0&1\\0&0\end{matrix}\right)$,
&
$\left(\begin{matrix} 0&0\\1&0\end{matrix}\right)$,
&
$\left(\begin{matrix} 0&0\\0&1\end{matrix}\right)$,
&
$\left(\begin{matrix} 1&0\\0&1\end{matrix}\right)$.
\end{tabular}
\]
The Brandt monoid $B_2^1$ is perhaps the most ubiquitous harbinger of complex behaviour in all finite semigroups. 
In particular, $B_2^1$ has no finite basis for its identities~\cite{Perkins-69} and is one of the four smallest semigroups with this property~\cite{Lee-Zhang-15}. 
It generates a monoid variety with uncountably many subvarieties~\cite{Jackson-Lee-18,Jackson-Zhang-21} and, moreover, it is the smallest generator for a monoid variety with uncountably many subvarieties~\cite{Gusev-Li-Zhang-23,Lee-Zhang-14}.

The 6-element monoid
\[
A_2^1:=\langle a, b \mid aba = a,\,bab = b,\,aa = 0,\, bb = b\rangle\cup\{1\}
\]
is one more of the most famous 6-element monoids.
It can be represented as the matrix semigroup
\[
\begin{tabular}{cccccc}
$\left(\begin{matrix} 0&0\\0&0\end{matrix}\right)$,
&
$\left(\begin{matrix} 0&1\\1&0\end{matrix}\right)$,
&
$\left(\begin{matrix} 1&0\\1&0\end{matrix}\right)$,
&
$\left(\begin{matrix} 1&0\\1&0\end{matrix}\right)$,
&
$\left(\begin{matrix} 0&1\\0&1\end{matrix}\right)$,
&
$\left(\begin{matrix} 1&0\\0&1\end{matrix}\right)$.
\end{tabular}
\]
It is well known that $A_2^1$ generates a variety properly containing that generated by $B_2^1$.
The monoid $A_2^1$ as well as the 6-element Brandt monoid $B_2^1$ plays a critical role in the theory of semigroup varieties. 
So, the following question is of fundamental interest.

\begin{problem}[{\!\cite[Question~6.2]{Gusev-Lee-20}}]
\label{prob:B21-A21}
Which, if any, of the monoids $B_2^1$ and $A_2^1$ generates a finitely universal variety?
\end{problem}

Problems~\ref{prob:min-order} and~\ref{prob:B21-A21} are addressed in the present article.
We exhibit a finitely universal monoid variety $\mathbb C$ and show that $\mathbb C$ is contained in the variety generated by the Brandt monoid $B_2^1$. 
Problems~\ref{prob:min-order} and~\ref{prob:B21-A21} are thus completely solved.

The new finitely universal variety $\mathbb C$ allows us to construct examples of two small varieties of monoids with an incredibly complex join resulting in solving the following problem.

\begin{problem}[{\!\cite[Question~6.4]{Gusev-Lee-20}}]
\label{prob:join}
\quad
\begin{itemize}
\item[\textup{(i)}] Are there varieties of monoids $\mathbb V_1$ and $\mathbb V_2$ that are not finitely universal such that the join $\mathbb V_1 \vee \mathbb V_2$ is finitely universal?
\item[\textup{(ii)}] Are there small varieties of monoids $\mathbb V_1$ and $\mathbb V_2$ such that the join $\mathbb V_1 \vee \mathbb V_2$ is finitely universal?
\end{itemize}
\end{problem}

\begin{remark}
Problem~\ref{prob:join}(i) has an affirmative answer within the context of varieties of semigroups, that is, there are two semigroup varieties that are not finitely universal such that their join is finitely universal. 
However, one of these varieties is not small, so that they do not provide an affirmative answer to Problem~\ref{prob:join}(ii) within the context of varieties of semigroups; see Section~6.3 in~\cite{Gusev-Lee-20} for more details.
\end{remark}

In fact, we not only provide an affirmative solution to Problem~\ref{prob:join}, but establish a much stronger counterintuitive result. 
Namely, we prove that there are two Cross varieties of monoids, whose lattices of subvarieties are the 6-element and the 7-element chains, respectively, such that the join of these two varieties is finitely universal and contains uncountably many subvarieties.
Moreover, we construct infinitely examples of finitely universal varieties with uncountably many subvarieties which are the join of two Cross varieties.

The article consists of five sections. 
Background information and some basic results are first given in Section~\ref{sec:prelim}.
In Section~\ref{sec:main-results} we introduce the variety $\mathbb C$, which, as we show in Section~\ref{sec:proof}, turns out to be finitely universal (Theorem~\ref{thm:C}).
Then we formulate our main results announced above (Theorems~\ref{thm:M(xzytxy)veeN} and~\ref{thm:B21-A21}) and deduce them from Theorem~\ref{thm:C}.
Section~\ref{sec:proof} is the technical core of the article; it is devoted to the proof of Theorem~\ref{thm:C}. 
We prove Theorem~\ref{thm:C} by showing that the lattice $\mathfrak{Eq}(\mathsf{A})$ of equivalence relations on every sufficiently large finite set $\mathsf{A}$ is anti-isomorphic to some subinterval of the lattice $\mathfrak L(\mathbb C)$ of subvarieties.
In view of the well-known theorem of Pudl\'ak and T$\mathring{\text{u}}$ma~\cite{Pudlak-Tuma-80} stating that every finite lattice is embeddable in a lattice of equivalence relations on a finite set, Theorem~\ref{thm:C} thus holds.
The article ends with some open problems in Section~\ref{sec:problems}.

\section{Preliminaries} 
\label{sec:prelim}

Acquaintance with rudiments of universal algebra is assumed of the reader.
Refer to the monograph~\cite{Burris-Sankappanavar-81} for more information.

\subsection{Words, identities, and deduction}

Let $\mathscr X^\ast$ denote the free monoid over a countably infinite alphabet $\mathscr X$.
Elements of~$\mathscr X$ are called \textit{variables} and elements of~$\mathscr X^\ast$ are called \textit{words}.
The \textit{content} of a word $\mathbf w$, that is, the set of all variables occurring in $\mathbf w$ is denoted by $\mathsf{con}(\mathbf w)$.
For a word $\mathbf w$ and a variable $x$, let $\mathsf{occ}_x(\mathbf w)$ denote the number of occurrences of $x$ in $\mathbf w$. 
A variable $x$ is called \textit{simple} [\textit{multiple}] \textit{in a word} $\mathbf w$ if $\mathsf{occ}_x(\mathbf w)=1$ [respectively, $\mathsf{occ}_x(\mathbf w)>1$]. 
The set of all simple [multiple] variables of a word $\mathbf w$ is denoted by $\mathsf{sim}(\mathbf w)$ [respectively, $\mathsf{mul}(\mathbf w)$]. 
A non-empty word $\mathbf w$ is called \textit{linear} if $\mathsf{con}(\mathbf w)=\mathsf{sim}(\mathbf w)$.
For any $\mathscr A\subseteq \mathsf{con}(\mathbf w)$, let $\mathbf w(\mathscr A)$ denote the word obtained by applying the substitution that fixes the variables in $\mathscr A$ and assigns the empty word~$1$ to all other variables.
Further, for any $\mathscr A\subseteq \mathsf{con}(\mathbf w)$, let $\mathbf w_{\mathscr A}:=\mathbf w(\mathsf{con}(\mathbf w)\setminus\mathscr A)$.
The expression $_{i\mathbf w}x$ means the $i$th occurrence of a variable $x$ in a word $\mathbf w$. 
If the $i$th occurrence of $x$ precedes the $j$th occurrence of $y$ in a word $\mathbf w$, then we write $({_{i\mathbf w}x}) < ({_{j\mathbf w}y})$.

An identity is written as $\mathbf u \approx \mathbf v$, where $\mathbf u,\mathbf v \in \mathscr X^\ast$; it is \textit{nontrivial} if $\mathbf u \neq \mathbf v$.
A variety $\mathbb V$ \textit{satisfies} an identity $\mathbf u \approx \mathbf v$, if for any monoid $M\in \mathbb V$ and any substitution $\varphi\colon \mathscr X \to M$, the equality $\varphi(\mathbf u)=\varphi(\mathbf v)$ holds in $M$.
An identity $\mathbf u \approx \mathbf v$ is \textit{directly deducible} from an identity $\mathbf s \approx \mathbf t$ if there exist some words $\mathbf a,\mathbf b \in \mathscr X^\ast$ and substitution $\varphi\colon \mathscr X \to \mathscr X^\ast$ such that $\{ \mathbf u,\mathbf v \} = \big\{ \mathbf a\varphi(\mathbf s)\mathbf b,\mathbf a\varphi(\mathbf t)\mathbf b \big\}$.
A nontrivial identity $\mathbf u \approx \mathbf v$ is \textit{deducible} from a set~$\Sigma$ of identities if there exists some finite sequence $\mathbf u = \mathbf w_0, \mathbf w_1, \ldots, \mathbf w_m = \mathbf v$ of distinct words such that each identity $\mathbf w_i \approx \mathbf w_{i+1}$ is directly deducible from some identity in~$\Sigma$.

\begin{proposition}[{Birkhoff's Completeness Theorem for Equational Logic; see \cite[Theorem~II.14.19]{Burris-Sankappanavar-81}}] 
\label{prop:deduction}
Let~$\mathbb V$ be the variety defined by some set~$\Sigma$ of identities.
Then~$\mathbb V$ satisfies an identity $\mathbf u \approx \mathbf v$ if and only if $\mathbf u \approx \mathbf v$ is deducible from $\Sigma$.\qed
\end{proposition}

Two sets of identities~$\Sigma_1$ and~$\Sigma_2$ are \textit{equivalent} (within a variety $\mathbb V$) if~$\Sigma_1$ and~$\Sigma_2$ define the same variety (within $\mathbb V$).

\subsection{Factor monoids}

For any set $\mathscr W$ of words, the \textit{factor monoid of $\mathscr W$}, denoted by $M(\mathscr W)$, is the monoid that consists of all factors of~$\mathscr W$ and a zero element~$0$, with multiplication~$\cdot$ given by 
\[ 
\mathbf u \cdot \mathbf v := 
\begin{cases} 
\mathbf u\mathbf v & \text{if $\mathbf u\mathbf v$ is a factor of~$\mathbf w$}, \\ 0 & \text{otherwise}; 
\end{cases} 
\] 
the empty word~$1$ is the identity element of $M(\mathscr W)$.
A word~$\mathbf w$ is an \textit{isoterm} for a variety~$\mathbb V$ if~$\mathbb V$ violates any nontrivial identity of the form $\mathbf w \approx \mathbf w^\prime$.
Given any set $\mathscr W$ of words, let $\mathbb M(\mathscr W)$ denote the variety generated by the factor monoid $M(\mathscr W)$.
One advantage in working with factor monoids is the relative ease of checking if a variety $\mathbb M(\mathscr W)$ is contained in some given variety.

\begin{lemma}[{\!\cite[Lemma~3.3]{Jackson-05}}] 
\label{lem:M(W)}
For any variety~$\mathbb V$ and any set~$\mathscr W$ of word, the inclusion $\mathbb M(\mathscr W) \subseteq \mathbb V$ holds if and only if any word in~$\mathscr W$ is an isoterm for~$\mathbb V$.
\end{lemma}

\section{Main results}
\label{sec:main-results}

\subsection{The variety $\mathbb C$.}

Here we introduce the variety $\mathbb C$, which, as we prove in Section~\ref{sec:proof}, is finitely universal.
All other finitely universal varieties in this article contain it.
We need some notation.
We denote by $S_k$ the symmetric group on the set $\{1,2,\dots,k\}$. 
As usual, $S_k^n$ denote the $n$th direct power of $S_k$.
If $\xi\in S_k^n$, then we denote by $\xi_i$ the $i$th component of $\xi$.
For any $n\ge2$ and $\xi\in S_2^n$, we define the word:
\[
\mathbf w_{\xi} := \mathbf p\,\biggl(\prod_{i=1}^na_i\biggr)\,a\, \biggl(\prod_{i=1}^nx_{1\xi_i}^{(i)}x_{2\xi_i}^{(i)}\biggr)\,b\,\biggl(\prod_{i=1}^nb_i\biggr)\, \mathbf q\mathbf r,
\]
where
\begin{align}
\label{eq:p=}
&\mathbf p := \biggl(\prod_{i=1}^n z_it_i\biggr)\biggl(\prod_{i=1}^nz_i^{\prime}t_i^{\prime}\biggr)\biggl(\prod_{i=1}^nz_i^{\prime\prime}t_i^{\prime\prime}\biggr),\\ 
\label{eq:q=}
&\mathbf q := \biggl(\prod_{i=0}^n s_iy_i\biggr)t,\\
\label{eq:r=}
&\mathbf r :=by_0\biggl(\prod_{i=1}^n x_1^{(i)}\,z_i a_iz_i^{\prime} b_iz_i^{\prime\prime} \,x_2^{(i)}\, y_i\biggr)a.
\end{align}
For any $n\ge2$, we denote by $\mathscr W_n$ the set of all words of the form $\mathbf w_{\xi}$ with $\xi\in S_2^n$.
Evidently, $|\mathscr W_n|=|S_2^n|=2^n$.

\begin{theorem}
\label{thm:C}
The variety $\mathbb C:= \mathbb M(\{\mathscr W_n\mid n\ge2\})$ is finitely universal.
\end{theorem}

\begin{remark}
Recall that a variety is \textit{periodic} if it satisfies the identity $x^{m+k} \approx x^m$ for some $m,k \geq 1$; in this case, the number~$m$ is the \textit{index} of the variety.
Varieties of index~1 are \textit{completely regular}, that is, consist of unions of groups.
The lattice of subvarieties of every completely regular variety of semigroups and, therefore, monoids is modular and moreover, Arguesian; this fundamental result was established in three different ways by Pastijn~\cite{Pastijn-90,Pastijn-91} and Petrich and Reilly~\cite{Petrich-Reilly-90} (see also Section~5.3 in the survey~\cite{Gusev-Lee-Vernikov-22}).
Thus, varieties of index~1 are not finitely universal.
For each $m \ge 3$, an example of a finitely universal variety of index $m$ was found in~\cite{Gusev-Lee-20}.
As for varieties of index~2, a finitely universal example was unknown so far; see~\cite[Question~6.1]{Gusev-Lee-20} or~\cite[Question~4.10]{Gusev-Lee-Vernikov-22}.
It is easy to see that the variety $\mathbb C$ satisfies the identity $x^2\approx x^3$ and so is of index~2. 
Thus, Theorem~\ref{thm:C} provides an example of a finitely universal variety of monoids of index~2.
\end{remark}

The proof of Theorem~\ref{thm:C} is given in Section~\ref{sec:proof}.
For the rest of Section~\ref{sec:main-results}, we discuss our main results and show how to deduce them from Theorem~\ref{thm:C}.

\subsection{The join of two Cross varieties}

Let $\mathbb N$ denote the variety defined by the identities 
\begin{equation}
\label{eq:five identities}
x^2\approx x^3,\ x^2y\approx yx^2,\ xyxzx\approx x^2yz,\ xzxyty\approx xzyxty,\ xzytxy\approx xzytyx.
\end{equation}
It is verified in~\cite[Theorem~1.1]{Gusev-19} that the lattice $\mathfrak L(\mathbb M(xzytxy)\vee\mathbb N)$ is as shown in Fig.~\ref{fig:L(M(xzytxy)veeN)}, where $\mathbb T$ is the variety of all trivial monoids and the interval $[\mathbb M(xzytxy), \mathbb M(xzytxy)\vee\mathbb N]$ contains uncountably many varieties.
In particular, the lattices $\mathfrak L(\mathbb M(xzytxy))$ and $\mathfrak L(\mathbb N)$ are the 6-element and the 7-element chains, respectively.
The following counterintuitive result provides a complete solution to Problem~\ref{prob:join}.

\begin{figure}[htb]
\unitlength=1mm
\linethickness{0.4pt}
\begin{center}
\begin{picture}(60,80)
\put(7,63){\circle*{1.33}}
\put(17,53){\circle*{1.33}}
\put(17,73){\circle*{1.33}}
\put(27,3){\circle*{1.33}}
\put(27,13){\circle*{1.33}}
\put(27,23){\circle*{1.33}}
\put(27,33){\circle*{1.33}}
\put(27,43){\circle*{1.33}}
\put(27,63){\circle*{1.33}}
\put(37,53){\circle*{1.33}}
\gasset{AHnb=0,linewidth=0.4}
\drawline(27,3)(27,43)(7,63)(17,73)
\drawline(27,43)(37,53)(27,63)(17,53)
\put(22,68)
{
\begin{rotate}{-47}
\drawoval[fillgray=0.7](-0.85,-0.55,12.7,4,2)
\end{rotate}
}
\gasset{AHnb=1,AHLength=2,linewidth=0.4}
\drawline(36,76)(22,68)
\put(27,0){\makebox(0,0)[cc]{$\mathbb T$}}
\put(28,13){\makebox(0,0)[lc]{$\mathbb M(\emptyset)$}}
\put(28,23){\makebox(0,0)[lc]{$\mathbb M(x)$}}
\put(28,33){\makebox(0,0)[lc]{$\mathbb M(xy)$}}
\put(28,42){\makebox(0,0)[lc]{$\mathbb M(xyx)$}}
\put(6,63){\makebox(0,0)[rc]{$\mathbb N$}}
\put(17,77){\makebox(0,0)[cc]{$\mathbb M(xzytxy)\vee\mathbb N$}}
\put(16,52){\makebox(0,0)[rc]{$\mathbb M(xyzxty)$}}
\put(29,64){\makebox(0,0)[lc]{$\mathbb M(xyzxty,xzytxy)$}}
\put(38,53){\makebox(0,0)[lc]{$\mathbb M(xzytxy)$}}
\put(37,78){\makebox(0,0)[lc]{\emph{uncountably}}}
\put(37,74){\makebox(0,0)[lc]{\emph{many varieties}}}
\end{picture}
\end{center}
\caption{The lattice $\mathfrak L(\mathbb M(xzytxy)\vee\mathbb N)$}
\label{fig:L(M(xzytxy)veeN)}
\end{figure}
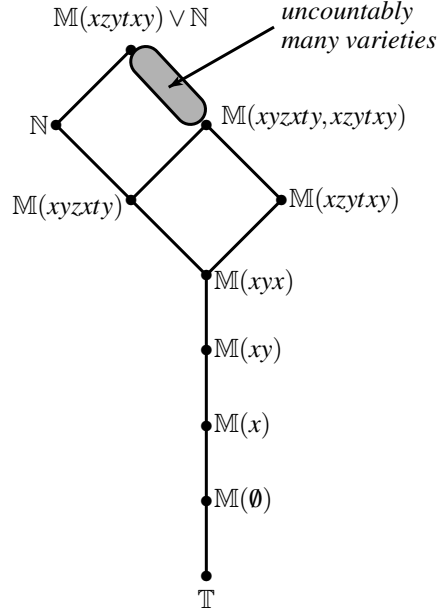

\begin{theorem}
\label{thm:M(xzytxy)veeN}
There are two Cross varieties of monoids such that the join of these varieties is finitely universal and contains uncountably many subvarieties. 
Namely, the varieties $\mathbb M(xzytxy)$ and $\mathbb N$ satisfy this property.
\end{theorem}

The proof of Theorem~\ref{thm:M(xzytxy)veeN} requires one intermediate result.

\begin{lemma}
\label{lem:FIC(M(xzytxy))-class}
Let $\mathbf u \approx \mathbf v$ be an identity of $\mathbb M(xzytxy)$.
Suppose that $\mathbf u \in \mathscr W_n$.
Then $\mathbf v=\mathbf p\mathbf v^\prime \mathbf q\mathbf r$, where the words $\mathbf p$, $\mathbf q$ and $\mathbf r$ are defined by the equalities~\eqref{eq:p=},~\eqref{eq:q=} and~\eqref{eq:r=}, respectively, while $\mathbf v^\prime$ is a linear word with $\mathsf{con}(\mathbf v^\prime)=\{a,a_i,b,b_i,x_1^{(i)},x_2^{(i)}\mid 1\le i\le n\}$.
\end{lemma}

\begin{proof}
In view of Lemma~3.1 in~\cite{Gusev-Vernikov-21}, $\mathbf v=\mathbf p\mathbf v^\prime \mathbf q\mathbf r^\prime$, where the words $\mathbf p$ and $\mathbf q$ are defined by the equalities~\eqref{eq:p=} and~\eqref{eq:q=}, respectively, while $\mathbf v^\prime$ and $\mathbf r^\prime$ are linear words with $\mathsf{con}(\mathbf v^\prime)=\{a,a_i,b,b_i,x_1^{(i)},x_2^{(i)}\mid 1\le i\le n\}$ and $\mathsf{con}(\mathbf r)=\mathsf{con}(\mathbf r^\prime)$.
Since $\mathbf u(b,s_0,t,y_0)=bs_0y_0tby_0$ and the identity $\mathbf u \approx \mathbf v$ is satisfied by $\mathbb M(xzytxy)$, the word $\mathbf v(b,s_0,t,y_0)$ must coincide with $bs_0y_0tby_0$. 
Therefore, $(_{1\mathbf r^\prime}b)<(_{1\mathbf r^\prime}y_0)$.
By a similar argument we can show that all the variables occur in $\mathbf r^\prime$ in the same order as in $\mathbf r$ and, therefore, $\mathbf r^\prime=\mathbf r$.
\end{proof}

\begin{proof}[Proof of Theorem~\ref{thm:M(xzytxy)veeN}]
It is shown in the Erratum to~\cite{Jackson-05} that the variety $\mathbb M(xzytxy)$ is finitely based.
In view of this fact and Fig.~\ref{fig:L(M(xzytxy)veeN)}, $\mathbb M(xzytxy)$ is a Cross variety.
A finite generator for the variety $\mathbb N$ is also exhibited in the Erratum to~\cite{Jackson-05}.
Thus, $\mathbb N$ is also a Cross variety.

Let $n\ge 2$, $\xi\in S_2^n$ and $\mathbf w_\xi \approx \mathbf w$ be an identity of $\mathbb M(xzytxy)\vee\mathbb N$.
It follows from Lemma~\ref{lem:FIC(M(xzytxy))-class} that $\mathbf w=\mathbf p\mathbf w^\prime \mathbf q\mathbf r$, where the words $\mathbf p$, $\mathbf q$ and $\mathbf r$ are defined by the equalities~\eqref{eq:p=},~\eqref{eq:q=} and~\eqref{eq:r=}, respectively, while $\mathbf w^\prime$ is a linear word with $\mathsf{con}(\mathbf w^\prime)=\{a,a_i,b,b_i,x_1^{(i)},x_2^{(i)}\mid 1\le i\le n\}$.
Further, consider arbitrary $x,y\in\mathsf{con}(\mathbf w^\prime)$ with $({_{1\mathbf w_\xi}x})<({_{1\mathbf w_\xi}y})$.
Then $\mathbf w_\xi(x,y,t)=xyt\mathbf a$ with $\mathbf a\in\{xy,yx\}$.
Since $\mathbb N$ violates $xyt\mathbf a \approx yxt\mathbf a$, it follows that $({_{1\mathbf w}x})<({_{1\mathbf w}y})$.
Therefore,
\[
\mathbf w^\prime=\biggl(\prod_{i=1}^na_i\biggr)\,a\, \biggl(\prod_{i=1}^nx_{1\xi_i}^{(i)}x_{2\xi_i}^{(i)}\biggr)\,b\,\biggl(\prod_{i=1}^nb_i\biggr)
\]
and so $\mathbf w=\mathbf w_\xi$.
We have proved that every word in $\mathscr W_n$ is an isoterm for $\mathbb M(xzytxy)\vee\mathbb N$. 
Thus, $\mathbb C\subseteq\mathbb M(xzytxy)\vee\mathbb N$ by Lemma~\ref{lem:M(W)}.
Now Theorem~\ref{thm:C} applies, yielding that the variety $\mathbb M(xzytxy)\vee\mathbb N$ is finitely universal.
Finally, $\mathbb M(xzytxy)\vee\mathbb N$ contains uncountably many subvarieties by~\cite[Theorem~1.1]{Gusev-19}.
\end{proof}

\begin{remark}
Theorem~\ref{thm:M(xzytxy)veeN} implies that a cover of a Cross variety of monoids can be finitely universal. 
In contrast, it is unknown whether or not the similar result holds within the context of varieties of semigroups.
Although it is known that the class of Cross semigroup varieties is closed under neither joins nor covers~\cite{Sapir-91}.
\end{remark}
 
For a monoid $K$, we denote by $\mathbb K$ the variety generated by $K$.
We have the following result on the join $\mathbb M(xyx) \vee \mathbb G$ for a group $G$ of finite exponent. 

\begin{corollary}
\label{cor:M(xyx)veeG}
Let $G$ be a group of finite exponent which does not satisfy the identities
\begin{equation}
\label{eq:two-identities}
xyzxy\approx yxzyx\  \text{ and }\ xyzyx\approx yxzxy.
\end{equation}
Then $\mathbb M(xyx) \vee \mathbb G$ is a non-finitely based finitely universal variety with uncountably many subvarieties.
\end{corollary}

The proof of Corollary~\ref{cor:M(xyx)veeG} requires one auxiliary result.

\begin{lemma}
\label{lem:M(xyzxy,xyzyx)}
The variety $\mathbb M(xzytxy)\vee\mathbb N$ is a subvariety of $\mathbb M(xyzxy,xyzyx)$.
\end{lemma}

\begin{proof}
Obviously, $\mathbb M(xzytxy)\subseteq\mathbb M(xyzxy,xyzyx)$.
Further, it is shown in the proof of Lemma~3.14 in~\cite{Gusev-Vernikov-21} that if a variety contains $\mathbb M(xyx)$ but does not contain $\mathbb N$, then it satisfies one of the identities $xyzxy\approx yxzxy$, $xyzxy\approx xyzyx$ or $xyzxy\approx yxzyx$.
Since these three identities do not hold in $\mathbb M(xyzxy,xyzyx)$, it follows that $\mathbb N\subseteq\mathbb M(xyzxy,xyzyx)$.
Therefore, $\mathbb M(xzytxy)\vee\mathbb N$ is a subvariety of $\mathbb M(xyzxy,xyzyx)$.
\end{proof}

\begin{proof}[Proof of Corollary~\ref{cor:M(xyx)veeG}]
Since the group $G$ does not satisfy the identities~\eqref{eq:two-identities}, this group is non-abelian, whence $\mathbb M(xyx) \vee \mathbb G$ is non-finitely based by~\cite[Theorem~3]{Lee-13}.
Let $xyzxy\approx \mathbf v$ be an identity of $\mathbb M(xyx) \vee \mathbb G$.
Since $xyx$ is an isoterm for $\mathbb M(xyx) \vee \mathbb G$, we have $\mathbf v\in\{xyzxy,xyzyx,yxzxy,yxzyx\}$.
The word $\mathbf v$ cannot coincide with $yxzyx$ because $G$ violates the identities~\eqref{eq:two-identities}.
Let $m$ denote the exponent of $G$.
If $\mathbf v=xyzyx$, then $G$ satisfies the identities
\[
xy\approx (xy)^{m+1}\approx (xy)^m(yx)\approx yx
\]
contradicting the fact that $G$ is a non-abelian group.
If $\mathbf v=yxzxy$, then $G$ satisfies the identities
\[
xy\approx (xy)^{m+1}\approx (yx)(xy)^m\approx yx
\]
contradicting the fact that the group $G$ is a non-abelian again.
Therefore, $\mathbf v=xyzxy$.
We see that $xyzxy$ is an isoterm for $\mathbb M(xyx) \vee \mathbb G$.
By similar arguments we can show that $xyzyx$ is an isoterm for $\mathbb M(xyx) \vee \mathbb G$ as well.
Now Lemmas~\ref{lem:M(W)} and~\ref{lem:M(xyzxy,xyzyx)} apply, yielding that $\mathbb M(xzytxy)\vee\mathbb N$ is a subvariety of $\mathbb M(xyx) \vee \mathbb G$.
Hence the variety $\mathbb M(xyx) \vee \mathbb G$ is finitely universal and contains uncountably many subvarieties by Theorem~\ref{thm:M(xzytxy)veeN}.
\end{proof}

\begin{remark}
By the theorem of Oates and Powell~\cite{Oates-Powell-64}, if $G$ is a finite group, then $\mathbb G$ is a Cross variety. 
Hence Corollary~\ref{cor:M(xyx)veeG} provides plenty of examples of non-finitely based finitely universal varieties of monoids with uncountably many subvarieties which are the join of two Cross varieties. 
Namely, they are the varieties of the form $\mathbb M(xyx)\vee\mathbb G$ for any finite group $G$ violated the identities~\eqref{eq:two-identities}. 
For example, for each prime $p>2$, consider the dihedral group
\[
D_p:=\langle a,b\mid a^p=b^2=(ab)^2=1\rangle
\]
(the group of symmetries of a regular polygon with $p$ sides).
This group does not satisfy the identities~\eqref{eq:two-identities}.
Indeed, consider the substitutions $\varphi\colon \mathscr X \to D_p$ and $\psi\colon \mathscr X \to D_p$ defined by
\[
\varphi(v):= 
\begin{cases} 
ab & \text{if $v=x$}, \\ 
b & \text{if $v=y$}, \\ 
1 & \text{otherwise},
\end{cases}
\ \ \text{ and }\ \ 
\psi(v):= 
\begin{cases} 
a & \text{if $v=x$}, \\ 
b & \text{if $v=y$}, \\ 
1 & \text{otherwise}.
\end{cases}
\]
It is routine to check that 
\[
\varphi(xyzxy)=\psi(xyzyx)=a^2 \ \text{ and }\ \varphi(yxzyx)=\psi(yxzxy)=a^{p-2}.
\]
Since $p>2$ and $p$ is prime, we have $a^2\ne a^{p-2}$.
Hence $D_p$ violates the identities~\eqref{eq:two-identities}.
It is well known that $D_p$ is a minimal non-abelian group of order $2p$; see~\cite[Section~1.9]{Hall-59}. 
From this it can be easily deduced that the lattice $\mathfrak L(\mathbb D_p)$ is as shown in Fig.~\ref{fig:L(D_p)} (we denote by $\mathbb Z_k$ the variety of all abelian groups of exponent dividing $k$).
Thus, we have a countably infinite series of finitely universal varieties
of monoids with uncountably many subvarieties which are the join of two Cross varieties
whose lattices of subvarieties are 5-element.
\end{remark}

\begin{figure}[htb]
\unitlength=1mm
\linethickness{0.4pt}
\begin{center}
\begin{picture}(30,36)
\put(15,3){\circle*{1.33}}
\put(5,13){\circle*{1.33}}
\put(25,13){\circle*{1.33}}
\put(15,23){\circle*{1.33}}
\put(15,33){\circle*{1.33}}
\gasset{AHnb=0,linewidth=0.4}
\drawline(15,3)(25,13)(15,23)(15,33)
\drawline(15,3)(5,13)(15,23)
\put(15,0){\makebox(0,0)[cc]{$\mathbb T$}}
\put(15,36){\makebox(0,0)[cc]{$\mathbb D_p$}}
\put(16.5,23){\makebox(0,0)[lc]{$\mathbb Z_{2p}$}}
\put(26,13){\makebox(0,0)[lc]{$\mathbb Z_p$}}
\put(4,13){\makebox(0,0)[rc]{$\mathbb Z_2$}}
\end{picture}
\end{center}
\caption{The lattice $\mathfrak L(\mathbb D_p)$}
\label{fig:L(D_p)}
\end{figure}
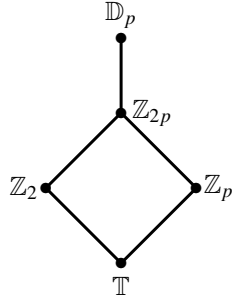

\begin{remark}
The following is claimed in the proof of Corollary~3.1 in~\cite{Gusev-19}:
$\mathbb M(xzytxy)\vee\mathbb N$ is a subvariety of $\mathbb M(xyx) \vee \mathbb G$ for any finite non-abelian group $G$.
In fact, this result is wrong in general. 
For example, it is easy to see that the quaternion group
\[
Q_8:=\langle i, j,k \mid i^2=j^2=k^2 = ijk\rangle
\]
satisfies the identities~\eqref{eq:two-identities}, whence $\mathbb N\nsubseteq \mathbb M(xyx)\vee \mathbb Q_8$.
As we have shown in the proof of Corollary~\ref{cor:M(xyx)veeG}, the discussed result is true whenever $G$ is a finite group violated the identities~\eqref{eq:two-identities}.
\end{remark}

\subsection{Minimal monoids generating finitely universal varieties}

Here we provide a complete solution to Problems~\ref{prob:min-order} and~\ref{prob:B21-A21}. 
As we have mentioned in the introduction, every monoid of order five or less generates a non-finitely universal variety~\cite[Proposition~6.9]{Gusev-Li-Zhang-23}.
Examples of finitely universal varieties generated by 6-element monoids are provided by the following theorem.

\begin{theorem}
\label{thm:B21-A21}
The 6-element monoids $B_2^1$ and $A_2^1$ generate finitely universal varieties.
\end{theorem}

\begin{proof}
It is easy to show that $xyzxy$ and $xyzyx$ are isoterms for $\mathbb B_2^1$; see the proof of Theorem~10 in~\cite{Perkins-69}.
This fact an Lemma~\ref{lem:M(W)} imply that $\mathbb M(xyzxy,xyzyx)\subseteq\mathbb B_2^1$.
Then $\mathbb M(xzytxy)\vee\mathbb N$ is a subvariety of $\mathbb B_2^1$ by Lemma~\ref{lem:M(xyzxy,xyzyx)} and, therefore, the variety $\mathbb B_2^1$ is finitely universal by Theorem~\ref{thm:M(xzytxy)veeN}.
Since $\mathbb A_2^1$ properly contains $\mathbb B_2^1$, the variety $\mathbb A_2^1$ is also finitely universal.
\end{proof}

\subsection{Finitely generated finitely based varieties}

Recall that a variety is \textit{locally finite} if every finitely generated member of it is finite.
More than being just non-finitely based, the varieties $\mathbb B_2^1$ and $\mathbb A_2^1$ are \textit{inherently non-finitely based} in the sense that every locally finite variety containing it is non-finitely based~\cite{Sapir-87}.
However, it is verified in~\cite[Theorem~3.2]{Jackson-Sapir-00} that the variety $\mathbb M(xyzxy,xyzyx)$ is finitely based by the first four identities in~\eqref{eq:five identities}.
Hence Theorem~\ref{thm:M(xzytxy)veeN} and Lemma~\ref{lem:M(xyzxy,xyzyx)} imply the following result providing an affirmative answer to Question~6.3 in~\cite{Gusev-Lee-20}.

\begin{theorem}
\label{thm:FB-FU}
There is a finitely universal variety of monoids that is both finitely based and finitely generated.\qed
\end{theorem}

\section{Proof of Theorem~\ref{thm:C}}
\label{sec:proof}

We verify Theorem~\ref{thm:C} modulo Proposition~\ref{prop:mapping} below and then prove this proposition.

\begin{proof}[Proof of Theorem~\ref{thm:C}]
The inclusion $\mathbb M(\mathscr W_n)\subseteq\mathbb C$ and Proposition~\ref{prop:mapping} imply that the lattice $\mathfrak{Eq}(\mathscr W_n)$ of equivalence relations on the set $\mathscr W_n$ is anti-isomorphic to a sublattice of $\mathfrak L(\mathbb C)$.
Since $|\mathscr W_n|=2^n$, it is easy to see that, for any $k=1,2,\dots,2^n$, the lattice $\mathfrak{Eq}(\{1,2,\ldots,k\})$ is anti-isomorphic to a sublattice of $\mathfrak{Eq}(\mathscr W_n)$.
Therefore, the lattice $\mathfrak L(\mathbb C)$ contains an anti-isomorphic copy of every finite lattice of equivalence relations.
To complete the proof, it remains to refer to the theorem of Pudl\'ak and T$\mathring{\text{u}}$ma~\cite{Pudlak-Tuma-80}, which states that every finite lattice can be embedded in a finite lattice of equivalence relations.
The variety $\mathbb C$ is thus finitely universal.
\end{proof}

The subvariety of a variety~$\mathbb V$ defined by a set~$\Sigma$ of identities is denoted by~$\mathbb V\Sigma$.
Given any set $\mathscr W \subseteq \mathscr X^\ast$ of words and any equivalence relation $\pi \in \mathfrak{Eq}(\mathscr W)$, define 
\[ 
\mathsf{Id}(\pi) := \{ \mathbf u \approx \mathbf v \mid (\mathbf u,\mathbf v) \in \pi \}. 
\]
For any set~$\mathsf{A}$, the universal relation on~$\mathsf{A}$ is denoted by~$\upsilon_\mathsf{A}$.

\begin{proposition}
\label{prop:mapping}
For each $n\ge 2$, the lattice $\mathfrak{Eq}(\mathscr W_n)$ is anti-isomorphic to the interval $[\mathbb M(\mathscr W_n)\{\mathsf{Id}(\upsilon_{\mathscr W_n})\},\mathbb M(\mathscr W_n)]$ of the lattice $\mathfrak L(\mathbb M(\mathscr W_n))$.
\end{proposition}

The proof of Proposition~\ref{prop:mapping} requires some intermediate results.

\begin{lemma}
\label{lem:isotrems-for-M(W_n)Id(U_n)}
For each $n\ge2$, the words $xy$, $xyx$, $xyzxty$, $xzytxy$ and $xytzsxzy$ are isoterms for the variety $\mathbb M(\mathscr W_n)\{\mathsf{Id}(\upsilon_{\mathscr W_n})\}$.
\end{lemma}

\begin{proof}
Consider the substitution $\varphi\colon \mathscr X \to \mathscr X^\ast$ defined by
\[
\varphi(v):= 
\begin{cases} 
x_2^{(1)} & \text{if $v=x$}, \\ 
x_1^{(2)} & \text{if $v=y$}, \\ 
y_1 & \text{if $v=z$}, \\ 
x_2^{(2)}\,x_{1}^{(3)}x_{2}^{(3)}x_{1}^{(4)}x_{2}^{(4)}\cdots x_{1}^{(n)}x_{2}^{(n)}\,bb_1b_2\cdots b_n\, s_0y_0s_1 & \text{if $v=t$}, \\
s_2y_2s_3y_3\cdots s_ny_n\,t b y_0\, x_1^{(1)}z_1 a_1z_1^{\prime} b_1z_1^{\prime\prime} & \text{if $v=s$}, \\ 
v & \text{otherwise}.
\end{cases} 
\]
Obviously, $\varphi(xytzsxzy)$ is a factor of $\mathbf w_\varepsilon$, where $\varepsilon$ is the identity element of $S_2^n$.
It follows that a nontrivial identity of the form $xytzsxzy \approx \mathbf a$ implies a nontrivial identity $\mathbf w_\varepsilon\approx \mathbf w$.
Therefore, $xytzsxzy$ is an isoterm for $\mathbb M(\mathscr W_n)$.
Further, it is routine to check that $M(xytzsxzy)$ satisfies $\mathbf w_{\xi}\approx \mathbf w_{\eta}$ for any $\xi,\eta\in S_2^n$.
Hence $xytzsxzy$ is an isoterm for $\mathbb M(\mathscr W_n)\{\mathsf{Id}(\upsilon_{\mathscr W_n})\}$ by Lemma~\ref{lem:M(W)}.
By similar arguments we can show that the words $xy$, $xyx$, $xyzxty$, $xzytxy$ are isoterms for $\mathbb M(\mathscr W_n)\{\mathsf{Id}(\upsilon_{\mathscr W_n})\}$ as well.
\end{proof}

\begin{lemma}
\label{lem:FIC(M(W_n)Id(U_n))-class}
Let $n\ge2$ and $\mathbf u \approx \mathbf v$ be an identity of $\mathbb M(\mathscr W_n)\{\mathsf{Id}(\upsilon_{\mathscr W_n})\}$.
Suppose that $\mathbf u \in \mathscr W_n$.
Then $\mathbf v \in \mathscr W_n$.
\end{lemma}

\begin{proof}
According to Lemma~\ref{lem:isotrems-for-M(W_n)Id(U_n)}, the words $xytzsxzy$ and $xzytxy$ are isoterms for the variety $\mathbb M(\mathscr W_n)\{\mathsf{Id}(\upsilon_{\mathscr W_n})\}$.
Then Lemmas~\ref{lem:M(W)} and~\ref{lem:FIC(M(xzytxy))-class} apply, yielding that $\mathbf v=\mathbf p\mathbf v^\prime \mathbf q\mathbf r$, where the words $\mathbf p$, $\mathbf q$ and $\mathbf r$ are defined by the equalities~\eqref{eq:p=},~\eqref{eq:q=} and~\eqref{eq:r=}, respectively, while $\mathbf v^\prime$ is a linear word with $\mathsf{con}(\mathbf v^\prime)=\{a,a_i,b,b_i,x_1^{(i)},x_2^{(i)}\mid 1\le i\le n\}$.
Further, $(_{1\mathbf v}a_i)<(_{1\mathbf v}a_{i+1})$ for any $i=1,2,\dots,n-1$ since $\mathbf u(a_i,a_{i+1},s_i,t,y_i)$ coincides (up to renaming of variables) with the word $xytzsxzy$ which is an isoterm for $\mathbb M(\mathscr W_n)\{\mathsf{Id}(\upsilon_{\mathscr W_n})\}$.
By a similar argument we can show that 
\begin{align*}
&(_{1\mathbf v}a_n)<(_{1\mathbf v}a),\ (_{1\mathbf v}a)<(_{1\mathbf v}x_1^{(1)}),\ (_{1\mathbf v}a)<(_{1\mathbf v}x_2^{(1)}),\\
&(_{1\mathbf v}x_1^{(i)})<(_{1\mathbf v}x_1^{(i+1)}),\ (_{1\mathbf v}x_1^{(i)})<(_{1\mathbf v}x_2^{(i+1)}),\\
&(_{1\mathbf v}x_2^{(i)})<(_{1\mathbf v}x_1^{(i+1)}),\ (_{1\mathbf v}x_2^{(i)})<(_{1\mathbf v}x_2^{(i+1)}),\\
&(_{1\mathbf v}x_1^{(n)})<(_{1\mathbf v}b),\ (_{1\mathbf v}x_2^{(n)})<(_{1\mathbf v}b),\ (_{1\mathbf v}b)<(_{1\mathbf v}b_1),\ (_{1\mathbf v}b_i)<(_{1\mathbf v}b_{i+1})
\end{align*}
for any $i=1,2,\dots,n-1$.
It follows that $\mathbf v =\mathbf w_{\eta}$ for some $\eta\in S_2^n$ and so $\mathbf v\in\mathscr W_n$.
\end{proof}

\begin{lemma}
\label{lem:directly}
Let $n\ge 2$, $\xi,\zeta,\eta\in S_2^n$ and $\mathbf w\in\mathscr X^\ast$.
Assume that $\mathbf w_\zeta=\mathbf a\varphi(\mathbf w_\xi)\mathbf b$ and $\mathbf w=\mathbf a\varphi(\mathbf w_\eta)\mathbf b$ for some words $\mathbf a,\mathbf b\in \mathscr X^\ast$ and substitution $\varphi\colon \mathscr X \to \mathscr X^\ast$.  
If the identity $\mathbf w_\zeta \approx \mathbf w$ is non-trivial, then $\varphi$ is the identity map on $\mathsf{con}(\mathbf w_\xi)$ and so $\mathbf a=\mathbf b=1$ and $(\mathbf w_\zeta,\mathbf w)=(\mathbf w_\xi,\mathbf w_\eta)$.
\end{lemma}

\begin{proof}
Since the identity $\mathbf w_{\zeta} \approx \mathbf w$ is nontrivial, Proposition~\ref{prop:deduction} and Lemma~\ref{lem:FIC(M(W_n)Id(U_n))-class} imply that $\mathbf w=\mathbf w_{\nu}$ for some $\nu\in S_2^n\setminus\{\zeta\}$.
Then there is $j\in\{1,2,\dots,n\}$ such that $(1\nu_j,2\nu_j)=(2\zeta_j,1\zeta_j)$.
This is only possible when $x_{1\zeta_j}^{(j)}\in\mathsf{con}(\varphi(x_{1\xi_k}^{(k)}))$ and $x_{2\zeta_j}^{(j)}\in\mathsf{con}(\varphi(x_{2\xi_k}^{(k)}))$ for some $k\in\{1,2,\dots,n\}$ with $(1\xi_k,2\xi_k)\ne(1\eta_k,2\eta_k)$.
We note that the word $\mathbf w_{\zeta}$ is square-free and every factor of length $>1$ of $\mathbf w_{\zeta}$ has exactly one occurrence in $\mathbf w_{\zeta}$.
It follows that
\begin{itemize}
\item[\textup{($\ast$)}] $\varphi(c)$ is either the empty word~$1$ or a variable for any $c\in\mathsf{mul}(\mathbf w_{\zeta})$.
\end{itemize}
In view of this fact, $x_{1\zeta_j}^{(j)}=\varphi(x_{1\xi_k}^{(k)})$ and $x_{2\zeta_j}=\varphi(x_{2\xi_k}^{(k)})$.
Further, since $(_{2\mathbf w_{\zeta}}x_1^{(j)})<(_{2\mathbf w_{\zeta}}x_2^{(j)})$ and $(_{2\mathbf w_{\xi}}x_1^{(k)})<(_{2\mathbf w_{\xi}}x_2^{(k)})$, we have $x_1^{(j)}=\varphi(x_1^{(k)})$ and $x_2=\varphi(x_2^{(k)})$ (this means that $\xi_k=\zeta_j$).
Hence
\[
\varphi(z_k a_k z_k^\prime b_kz_k^{\prime\prime})=z_j a_j z_j^\prime b_jz_j^{\prime\prime}.
\]
It follows from~($\ast$) that $\varphi(a_k)=a_j$, $\varphi(b_k)=b_j$, $\varphi(z_k)=z_j$, $\varphi(z_k^\prime)=z_j^\prime$ and $\varphi(z_k^{\prime\prime})=z_j^{\prime\prime}$.
Then 
\[
\varphi\biggl(\biggl(\prod_{i=k+1}^na_i\biggr)\,a\, \biggl(\prod_{i=1}^nx_{1\xi_i}^{(i)}x_{2\xi_i}^{(i)}\biggr)\,b\,\biggl(\prod_{i=1}^{k-1}b_i\biggr)\biggr)=\biggl(\prod_{i=j+1}^na_i\biggr)\,a\, \biggl(\prod_{i=1}^nx_{1\zeta_i}^{(i)}x_{2\zeta_i}^{(i)}\biggr)\,b\,\biggl(\prod_{i=1}^{j-1}b_i\biggr).
\]
If $k>j+1$, then $\varphi(b_{k-j})=x_{2\zeta_n}^{(n)}$ and $\varphi(b_{k-j+1})=b_1$ contradicting the fact that $(_{2\mathbf w_\xi}b_{k-j})<(_{2\mathbf w_\xi}b_{k-j+1})$ and $(_{2\mathbf w_\zeta}b_1)<(_{2\mathbf w_\zeta}x_{2\zeta_n}^{(n)})$. 
If $k=j+1$, then $\varphi(b)=x_{2\zeta_n}^{(n)}$ and $\varphi(b_1)=b$ contradicting the fact that $(_{2\mathbf w_\xi}b)<(_{2\mathbf w_\xi}b_1)$ and $(_{2\mathbf w_\zeta}b_1)<(_{2\mathbf w_\zeta}x_{2\zeta_n}^{(n)})$. 
Hence $k\le j$.
By a similar argument we can show that $j\le k$ and, therefore, $k=j$.
Then $\varphi(x_{1\xi_i}^{(i)})=x_{1\zeta_i}^{(i)}$ and $\varphi(x_{2\xi_i}^{(i)})=x_{2\zeta_i}^{(i)}$ for any $i=1,2,\dots,n$ by~($\ast$).
Since $(_{2\mathbf w_{\zeta}}x_1^{(i)})<(_{2\mathbf w_{\zeta}}x_2^{(i)})$ and $(_{2\mathbf w_{\xi}}x_1^{(i)})<(_{2\mathbf w_{\xi}}x_2^{(i)})$, this implies that $\xi=\zeta$ and
\[
\varphi(z_i a_i z_i^\prime b_iz_i^{\prime\prime})=z_i a_i z_i^\prime b_iz_i^{\prime\prime}
\]
for any $i=1,2,\dots,n$.
Now~($\ast$) applies again, yielding that $\varphi(a_i)=a_i$, $\varphi(b_i)=b_i$, $\varphi(z_i)=z_i$, $\varphi(z_i^\prime)=z_i^\prime$, $\varphi(z_i^{\prime\prime})=z_i^{\prime\prime}$ for any $i=1,2,\dots,n$.
It follows that $\varphi(a)=a$, $\varphi(b)=b$ and $\varphi(y_i)=y_i$ for all $i=1,2,\dots,n-1$.
Hence $\varphi(y_0)=y_0$, $\varphi(y_n)=y_n$ and so $\varphi(t)=t$, $\varphi(t_i)=t_i$, $\varphi(t_i^\prime)=t_i^\prime$, $\varphi(t_i^{\prime\prime})=t_i^{\prime\prime}$, $\varphi(s_i)=s_i$.
Thus, $\varphi$ is the identity map on $\mathsf{con}(\mathbf w_\xi)$.
Hence $\mathbf a=\mathbf b=1$ and so $(\mathbf w_\zeta,\mathbf w)=(\varphi(\mathbf w_\xi),\varphi(\mathbf w_\eta))=(\mathbf w_\xi,\mathbf w_\eta)$ as required.
\end{proof}

\begin{corollary}
\label{cor:FIC(M(W_n)Id(pi))-class}
Let $n\ge2$ and $\mathbf u \approx \mathbf v$ be an identity of $\mathbb M(\mathscr W_n)\{\mathsf{Id}(\pi)\}$ for some $\pi \in \mathfrak{Eq}(\mathscr W_n)$.
Suppose that $\mathbf u \in \mathscr W_n$.
Then $\mathbf v \in \mathscr W_n$ and $(\mathbf u,\mathbf v) \in \pi$.
\end{corollary}

\begin{proof}
In view of Proposition~\ref{prop:deduction}, there is some finite sequence $\mathbf u = \mathbf v_0, \mathbf v_1, \ldots, \mathbf v_m = \mathbf v$ of distinct words such that each identity $\mathbf v_i \approx \mathbf v_{i+1}$ is either holds in $\mathbb M(\mathscr W_n)$ or directly deducible from some identity in~$\mathsf{Id}(\pi)$.
According to Lemma~\ref{lem:FIC(M(W_n)Id(U_n))-class}, the word $\mathbf v_i$ belongs to $\mathscr W_n$ for any $i=0,1,\dots,m$.
Then Lemma~\ref{lem:M(W)} and the fact that the words $\mathbf v_0,\mathbf v_1,\dots,\mathbf v_m$ are pairwise distinct imply that $\mathbb M(\mathscr W_n)$ violates $\mathbf v_i \approx \mathbf v_{i+1}$ for any $i=0,1,\dots,m-1$.
Therefore, $\mathbf v_i \approx \mathbf v_{i+1}$ is directly deducible from some identity in~$\mathsf{Id}(\pi)$.
Now Lemma~\ref{lem:directly} applies, yielding that $(\mathbf v_i, \mathbf v_{i+1})\in \pi$, whence $(\mathbf u,\mathbf v) \in \pi$.
\end{proof}

\begin{lemma}
\label{lem:three-isotrems}
Let $n\ge2$ and $\zeta\in S_2^n$.
A word $\mathbf w$ is an isoterm for the variety $\mathbb M(\mathscr W_n)\{\mathsf{Id}(\upsilon_{\mathscr W_n})\}$ if one of the following holds:
\begin{itemize}
\item[\textup{(i)}] $\mathbf w$ is obtained from $\mathbf w_\zeta$ by replacing some occurrence of a multiple variable with a variable $h\notin\mathsf{con}(\mathbf w_\zeta)$;
\item[\textup{(ii)}] $\mathbf w$ is obtained from $\mathbf w_\zeta$ by replacing some factor of length $>1$ with a variable $h\notin\mathsf{con}(\mathbf w_\zeta)$;
\item[\textup{(iii)}] $\mathbf w$ is a proper factor of $\mathbf w_\zeta$.
\end{itemize}
\end{lemma}

\begin{proof}
(i) The word $\mathbf w$ is obtained from $\mathbf w_\zeta$ by replacing some occurrence of a multiple variable $c$ with the variable $h\notin\mathsf{con}(\mathbf w_\zeta)$. 
Clearly, $\psi(\mathbf w)=\mathbf w_\zeta$, where $\psi\colon \mathscr X \to \mathscr X^\ast$ is the substitution defined by
\[
\psi(v):= 
\begin{cases} 
c & \text{if $v=h$}, \\ 
v & \text{if $v\ne h$}.
\end{cases} 
\]
Since $xy$ is an isoterm for $\mathbb M(\mathscr W_n)$ by Lemma~\ref{lem:isotrems-for-M(W_n)Id(U_n)} and $c,h\in\mathsf{sim}(\mathbf w)$, it follows that $\mathbf w$ is an isoterm for $\mathbb M(\mathscr W_n)$.
Hence, by Proposition~\ref{prop:deduction}, if $\mathbf w$ is not an isoterm for the variety $\mathbb M(\mathscr W_n)\{\mathsf{Id}(\upsilon_{\mathscr W_n})\}$, then some nontrivial identity $\mathbf w\approx \mathbf w^\prime$ is directly deducible from some identity of the form $\mathbf w_\xi\approx \mathbf w_\eta$.
By symmetry, we may assume that $\mathbf w=\mathbf a\varphi(\mathbf w_{\xi})\mathbf b$ and $\mathbf w^\prime=\mathbf a\varphi(\mathbf w_{\eta})\mathbf b$ for some words $\mathbf a,\mathbf b\in \mathscr X^\ast$ and substitution $\varphi\colon \mathscr X \to \mathscr X^\ast$. 
Hence $\mathbf w_\zeta=\psi(\mathbf w)=\psi(\mathbf a)\psi(\varphi(\mathbf w_\xi))\psi(\mathbf b)$.
Then $\psi(\mathbf w^\prime)\ne\psi(\mathbf w)$ because $xy$ is an isoterm for the variety defined by the identity $\mathbf w_\xi\approx \mathbf w_\eta$ by Lemma~\ref{lem:isotrems-for-M(W_n)Id(U_n)}, $c,h\in\mathsf{sim}(\mathbf w)$ and the identity $\mathbf w\approx \mathbf w^\prime$ is nontrivial.
Now Lemma~\ref{lem:directly} applies, yielding that the substitution $\psi\varphi$ is the identity map on $\mathsf{con}(\mathbf w_\xi)$ and so $\psi(\mathbf a)=\psi(\mathbf b)=1$ and $\mathbf w_\zeta=\mathbf w_\xi$.
Then $\mathbf a=\mathbf b=1$ by the definition of the substitution $\psi$.
Thus, $\mathbf w=\varphi(\mathbf w_\xi)$. 
Since $c\in\mathsf{sim}(\mathbf w)$, there is $c^\prime \in \mathsf{sim}(\mathbf w_\xi)$ such that $c\in\mathsf{con}(\varphi(c^\prime))$.
Clearly, $\psi(c)=c$.
Hence $c\in\mathsf{con}(\psi(\varphi(c^\prime)))$.
Since the substitution $\psi\varphi$ is the identity map on $\mathsf{con}(\mathbf w_\xi)$, we have $c=c^\prime$ contradicting the fact that $c\in\mathsf{mul}(\mathbf w_\zeta)=\mathsf{mul}(\mathbf w_\xi)$ and $c^\prime\in\mathsf{sim}(\mathbf w_\xi)$.

\smallskip

(ii) The word $\mathbf w$ is obtained from $\mathbf w_\zeta$ by replacing some factor $cd$ with the variable $h\notin\mathsf{con}(\mathbf w_\zeta)$. 
Since every factor of length $>1$ of $\mathbf w_\zeta$ contains a multiple variable, we may assume without any loss that $c\in \mathsf{mul}(\mathbf w_\zeta)$.
Then the word $\psi(\mathbf w)$ is obtained from $\mathbf w_\zeta$ by replacing an occurrence of $c$ with the variable $h$, where $\psi\colon \mathscr X \to \mathscr X^\ast$ is the substitution defined by
\[
\psi(v):= 
\begin{cases} 
hd & \text{if $v=h$}, \\ 
v & \text{if $v\ne h$}.
\end{cases} 
\]
By Part~(i), the word $\psi(\mathbf w)$ is an isoterm for $\mathbb M(\mathscr W_n)\{\mathsf{Id}(\upsilon_{\mathscr W_n})\}$.
Hence $\mathbf w$ is an isoterm for $\mathbb M(\mathscr W_n)\{\mathsf{Id}(\upsilon_{\mathscr W_n})\}$ as well.

\smallskip

(iii) Let $\mathbf w_1$ and $\mathbf w_2$ denote words obtained from $\mathbf w_\zeta$ by replacing the variables ${_{1\mathbf w_\zeta}z_1}$ and ${_{2\mathbf w_\zeta}a}$ with the variable $h$, respectively.
Clearly, if some proper factor of $\mathbf w_\zeta$ is not an isoterm for $\mathbb M(\mathscr W_n)\{\mathsf{Id}(\upsilon_{\mathscr W_n})\}$, then at least one of the words $\mathbf w_1$ or $\mathbf w_2$ is not an isoterm for $\mathbb M(\mathscr W_n)\{\mathsf{Id}(\upsilon_{\mathscr W_n})\}$ as well.
Thus, $\mathbf w$ is an isoterm for $\mathbb M(\mathscr W_n)\{\mathsf{Id}(\upsilon_{\mathscr W_n})\}$ by Part~(i).
\end{proof}

\begin{lemma}
\label{lem:u_C}
Let $n\ge2$ and $\mathbf u$ be a word such that $\mathbf u_{\mathscr C}=\mathbf w_\zeta$ for some $\zeta\in S_2^n$ and $\mathscr C\subseteq\mathsf{con}(\mathbf u)$.
Assume that the following three claims hold:
\begin{itemize}
\item[\textup{(a)}] every factor of length $>1$ of $\mathbf u$ has exactly one occurrence in $\mathbf u$;
\item[\textup{(b)}]  there are no simple variables between ${_{1\mathbf u}a_1}$ and ${_{1\mathbf u}b_n}$ and between ${_{2\mathbf u}b}$ and ${_{2\mathbf u}a}$ in $\mathbf u$;
\item[\textup{(c)}] for some $c\in\mathscr C$, either $({_{1\mathbf u}a_1})<({_{1\mathbf u}c})<({_{1\mathbf u}b_n})$ or $({_{2\mathbf u}b})<({_{2\mathbf u}c})<({_{2\mathbf u}a})$.
\end{itemize}
If $\mathbf u\approx \mathbf v$ is a nontrivial identity directly deducible from some identity of the form $\mathbf w_\xi\approx \mathbf w_\eta$ with $\xi,\eta\in S_2^n$, then $\mathbf v_{\mathscr C}=\mathbf w_\zeta$.
\end{lemma}

\begin{proof}
By symmetry, we may assume that $\mathbf u=\mathbf a\varphi(\mathbf w_{\xi})\mathbf b$ and $\mathbf v=\mathbf a\varphi(\mathbf w_{\eta})\mathbf b$ for some words $\mathbf a,\mathbf b\in \mathscr X^\ast$ and substitution $\varphi\colon \mathscr X \to \mathscr X^\ast$. 
Then $\mathbf w_\zeta=\psi(\mathbf u)=\psi(\mathbf a)\psi(\varphi(\mathbf w_\xi))\psi(\mathbf b)$, where $\psi\colon \mathscr X \to \mathscr X^\ast$ is the substitution defined by 
\[
\psi(v):= 
\begin{cases} 
1 & \text{if $v\in\mathscr C$}, \\ 
v & \text{if $v\notin\mathscr C$}.
\end{cases} 
\]
Arguing by contradiction, suppose that $\mathbf v_{\mathscr C}=\psi(\mathbf v)\ne\mathbf w_\zeta$.
Then, by Lemma~\ref{lem:directly}, the substitution $\psi\varphi$ is the identity map on $\mathsf{con}(\mathbf w_\xi)$ and so $\psi(\mathbf a)=\psi(\mathbf b)=1$ and $\mathbf w_\zeta=\mathbf w_\xi$.
Hence $\mathsf{con}(\mathbf a\mathbf b)\subseteq\mathscr C$. 
Assume that $({_{1\mathbf u}a_1})<({_{1\mathbf u}c})<({_{1\mathbf u}b_n})$ for some $c\in\mathscr C$.
Then there is $c^\prime\in\mathsf{con}(\mathbf w_\xi)$ such that $\varphi$ maps some occurrence of $c^\prime$ to a factor of $\mathbf u$ containing ${_{1\mathbf u}c}$. 
The fact that $\psi(\varphi(c^\prime))=c^\prime$ implies that $\varphi(c^\prime)$ is a word of length $>1$.
By the condition of the lemma, the word $\mathbf u$ may contain at most one occurrence of the factor $\varphi(c^\prime)$.
This only possible when $c^\prime \in \mathsf{sim}(\mathbf w_\xi)$.
Since there are no simple variables between ${_{1\mathbf u}a_1}$ and ${_{1\mathbf u}b_n}$ in $\mathbf u$, it follows that $\mathsf{con}(\varphi(c^\prime))$ must contain either $a_1$ or $b_n$ contradicting $\psi(\varphi(c^\prime))=c^\prime$.
Therefore, $\mathbf v_{\mathscr C}=\psi(\mathbf v)=\mathbf w_\zeta$.
By a similar argument we can show that if $({_{2\mathbf u}b})<({_{2\mathbf u}c})<({_{2\mathbf u}a})$ for some $c\in\mathscr C$, then $\mathbf v_{\mathscr C}=\psi(\mathbf v)=\mathbf w_\zeta$.
\end{proof}

A \textit{block} of a word $\mathbf w$ is a maximal factor of $\mathbf w$ that does not contain any variables simple in $\mathbf w$. 
A word $\mathbf w$ is called \textit{block-linear} if every block of $\mathbf w$ is a linear word.

\begin{lemma}
\label{lem:u_{c,h}}
Let $n\ge2$ and $\mathbf u$ be a block-linear word such that $\mathbf u_{\{c,h\}}=\mathbf w_\zeta$ for some $\zeta\in S_2^n$ and $c,h\in\mathscr X$ with $h\in\mathsf{sim}(\mathbf u)$ and $\mathsf{occ}_c(\mathbf u)=2$.
Assume that, for some $x,y\in \mathsf{mul}(\mathbf w_\zeta)$ and $i,j$ with $\{i,j\}=\{1,2\}$, the word ${_{i\mathbf u}x}\,{_{i\mathbf u}c}\,{_{i\mathbf u}y}$ is a factor of $\mathbf u$, while the word ${_{j\mathbf u}c}$ forms a block of $\mathbf u$.
If $\mathbf u\approx \mathbf v$ is a nontrivial identity of $\mathbb M(\mathscr W_n)\{\mathsf{Id}(\upsilon_{\mathscr W_n})\}$, then $\mathbf v_{\{c,h\}}=\mathbf w_\zeta$.
\end{lemma}

\begin{proof}
In view of Proposition~\ref{prop:deduction}, there is some finite sequence $\mathbf u = \mathbf v_0, \mathbf v_1, \ldots, \mathbf v_m = \mathbf v$ of distinct words such that each identity $\mathbf v_i \approx \mathbf v_{i+1}$ either holds in $\mathbb M(\mathscr W_n)$ or is directly deducible from some identity in~$\mathsf{Id}(\upsilon_{\mathscr W_n})$.
We will use induction on $m$.

\smallskip

\textbf{Induction base:} $m=1$. 
If $\mathbf u=\mathbf v_0 \approx \mathbf v_1=\mathbf v$ holds in $\mathbb M(\mathscr W_n)$, then the required claim follows from Lemma~\ref{lem:M(W)}.
If $\mathbf u=\mathbf v_0 \approx \mathbf v_1=\mathbf v$ is directly deducible from some identity in~$\mathsf{Id}(\upsilon_{\mathscr W_n})$, then the condition of the lemma implies that the conditions (a), (b) and (c) of Lemma~\ref{lem:u_C} holds. 
So, we can apply Lemma~\ref{lem:u_C}, yielding that $\mathbf v_{\{c,h\}}=\mathbf w_\zeta$.

\smallskip

\textbf{Induction step:} $m>1$. 
First, notice that, as in the induction base, $(\mathbf v_1)_{\{c,h\}}=\mathbf w_\zeta$ by Lemmas~\ref{lem:M(W)} and~\ref{lem:u_C}.
Since $xyx$ is an isoterm for the variety $\mathbb M(\mathscr W_n)\{\mathsf{Id}(\upsilon_{\mathscr W_n})\}$ by Lemma~\ref{lem:isotrems-for-M(W_n)Id(U_n)}, the word ${_{j\mathbf u}c}$ forms a block of $\mathbf v_1$. 
Hence ${_{j\mathbf v_1}x}$ and ${_{j\mathbf v_1}y}$ do not lie in the block of $\mathbf v_1$ containing ${_{j\mathbf v_1}c}$.
Then, since $xyzxty$ and $xzytxy$ are isoterms for $\mathbb M(\mathscr W_n)\{\mathsf{Id}(\upsilon_{\mathscr W_n})\}$ by Lemma~\ref{lem:isotrems-for-M(W_n)Id(U_n)}, $({_{i\mathbf v_1}x})<({_{i\mathbf v_1}c})<({_{i\mathbf v_1}y})$ and so the word ${_{i\mathbf v_1}x}\,{_{i\mathbf v_1}c}\,{_{i\mathbf v_1}y}$ is a factor of $\mathbf v_1$.
Thus, we can apply the induction assumption, yielding that $\mathbf v_{\{c,h\}}=\mathbf w_\zeta$ as required.
\end{proof}

\begin{lemma}
\label{lem:2x-2c-2y}
Let $n\ge2$ and $\mathbf u$ be a block-linear word such that $\mathbf u_c=\mathbf w_\zeta$ for some $\zeta\in S_2^n$ and $c\in\mathsf{mul}(\mathbf u)$ with $\mathsf{occ}_c(\mathbf u)=2$.
Assume that, for some $x,y\in \mathsf{mul}(\mathbf w_\zeta)$, the word ${_{2\mathbf u}x}\,{_{2\mathbf u}c}\,{_{2\mathbf u}y}$ is a factor of $\mathbf u$ and one of the following holds:
\begin{itemize}
\item[\textup{(i)}] $x\ne b$, $y\ne a$ and ${_{1\mathbf u}c}$ is not adjacent to ${_{1\mathbf u}x}$ and ${_{1\mathbf u}y}$ in $\mathbf u$;
\item[\textup{(ii)}] $x=b$, the variables ${_{1\mathbf u}c}$ and ${_{1\mathbf u}x}$ lie in different blocks of $\mathbf u$ and ${_{1\mathbf u}c}$ is not adjacent to ${_{1\mathbf u}y}$ in $\mathbf u$;
\item[\textup{(iii)}] $y=a$, the variables ${_{1\mathbf u}c}$ and ${_{1\mathbf u}y}$ lie in different blocks of $\mathbf u$  and ${_{1\mathbf u}c}$ is not adjacent to ${_{1\mathbf u}x}$ in $\mathbf u$.
\end{itemize}
If $\mathbf u\approx \mathbf v$ is a nontrivial identity of $\mathbb M(\mathscr W_n)\{\mathsf{Id}(\upsilon_{\mathscr W_n})\}$, then $\mathbf v_c=\mathbf w_\zeta$.
\end{lemma}

\begin{proof}
In view of Proposition~\ref{prop:deduction}, there is some finite sequence $\mathbf u = \mathbf v_0, \mathbf v_1, \ldots, \mathbf v_m = \mathbf v$ of distinct words such that each identity $\mathbf v_i \approx \mathbf v_{i+1}$ either holds in $\mathbb M(\mathscr W_n)$ or is directly deducible from some identity in~$\mathsf{Id}(\upsilon_{\mathscr W_n})$.
We will use induction on $m$.

\smallskip

\textbf{Induction base:} $m=1$. 
If $\mathbf u=\mathbf v_0 \approx \mathbf v_1=\mathbf v$ holds in $\mathbb M(\mathscr W_n)$, then the required claim follows from Lemma~\ref{lem:M(W)}.
If $\mathbf u=\mathbf v_0 \approx \mathbf v_1=\mathbf v$ is directly deducible from some identity in~$\mathsf{Id}(\upsilon_{\mathscr W_n})$, then the condition of the lemma implies that every factor of length $>1$ of $\mathbf u$ has exactly one occurrence in $\mathbf u$ and $({_{2\mathbf u}b})<({_{2\mathbf u}c})<({_{2\mathbf u}a})$.
Then we can apply Lemma~\ref{lem:u_C}, yielding that $\mathbf v_c=\mathbf w_\zeta$.

\smallskip

\textbf{Induction step:} $m>1$. 
First, notice that, as in the induction base, $(\mathbf v_1)_c=\mathbf w_\zeta$ by Lemmas~\ref{lem:M(W)} and~\ref{lem:u_C}.
By symmetry, it suffices to verify only Parts~(i) and~(ii).
The proof of Part~(ii) is very similar to the proof of Part~(i) but a bit simpler and we omit it.
So, we assume below that~(i) holds.

By symmetry, we may assume without any loss that $x\in\{a,a_i,b,b_i,x_1^{(i)},x_2^{(i)}\mid 1\le i\le n\}$ and $y\in\{y_0,y_i,z_i,z_i^\prime,z_i^{\prime\prime}\mid1\le i\le n\}$.
Then the variables ${_{1\mathbf u}c}$ and ${_{1\mathbf u}y}$ do not lie in the same block of $\mathbf u$ because these variables are not adjacent to each other in $\mathbf u$. 
The variables ${_{1\mathbf u}c}$ and ${_{2\mathbf u}c}$ also do not lie the same block of the word $\mathbf u$ because this word is block-linear.
Since $xzytxy$ and so $xzytyx$ are isoterms for $\mathbb M(\mathscr W_n)\{\mathsf{Id}(\upsilon_{\mathscr W_n})\}$ by Lemma~\ref{lem:isotrems-for-M(W_n)Id(U_n)} and $\mathsf{occ}_c(\mathbf u)=2$, this implies that $\mathsf{occ}_c(\mathbf v_1)=2$ and $({_{2\mathbf v_1}c})<({_{2\mathbf v_1}y})$.
Further, if ${_{1\mathbf u}c}$ and ${_{1\mathbf u}x}$ do not lie in the same block of $\mathbf u$, then $({_{2\mathbf v_1}x})<({_{2\mathbf v_1}c})$ and so ${_{2\mathbf v_1}x}\,{_{2\mathbf v_1}c}\,{_{2\mathbf v_1}y}$ is a factor of $\mathbf v_1$.
If ${_{1\mathbf u}c}$ and ${_{1\mathbf u}x}$ lie in the same block of $\mathbf u$, then $({_{2\mathbf v_1}z})<({_{2\mathbf v_1}c})$, where $z\in\{y_0,y_i,z_i,z_i^\prime,z_i^{\prime\prime}\mid 1\le i\le n\}$ is the variable such that ${_{2\mathbf w_\zeta}z}\,{_{2\mathbf w_\zeta}x}$ is a factor of $\mathbf w_\zeta$, and, therefore, either ${_{2\mathbf v_1}x}\,{_{2\mathbf v_1}c}\,{_{2\mathbf v_1}y}$ or ${_{2\mathbf v_1}z}\,{_{2\mathbf v_1}c}\,{_{2\mathbf v_1}x}$ is a factor of $\mathbf v_1$.

Suppose that ${_{1\mathbf v_1}c}$ is adjacent to ${_{1\mathbf v_1}x}$.
If $\mathbf u \approx \mathbf v_1$ holds in $\mathbb M(\mathscr W_n)$, then the word $(\mathbf v_1)_x$ coincides (up to renaming of variables) with $\mathbf w_\zeta$ and $\mathbf u_x\ne (\mathbf v_1)_x$ contradicting the fact that $\mathbf w_\zeta$ is an isoterm for $\mathbb M(\mathscr W_n)$. 
Therefore, $\mathbf u \approx \mathbf v_1$ is directly deducible from some identity $\mathbf w_\xi\approx \mathbf w_\eta$ in~$\mathsf{Id}(\upsilon_{\mathscr W_n})$. 
Then $(\mathbf v_1^\prime)_c=\mathbf w_\zeta\ne\mathbf u_c^\prime$, where $\mathbf u^\prime:=\psi(\mathbf u)$, $\mathbf v_1^\prime:=\psi(\mathbf v_1)$ and $\psi\colon \mathscr X \to \mathscr X^\ast$ is the substitution defined by
\[
\psi(v):= 
\begin{cases} 
x & \text{if $v=c$}, \\
c & \text{if $v=x$}, \\  
v & \text{otherwise}.
\end{cases} 
\] 
According to Lemma~\ref{lem:FIC(M(W_n)Id(U_n))-class}, there is $\nu\in S_2^n\setminus\{\zeta\}$ such that $\mathbf u_c^\prime=\mathbf w_\nu$.
In particular, ${_{1\mathbf u^\prime}x}$ and ${_{1\mathbf u^\prime}c}$ lie in the same block of $\mathbf u^\prime$.
Evidently, $\mathbf u^\prime\approx\mathbf v_1^\prime$ is directly deducible from $\mathbf w_\xi\approx \mathbf w_\eta$, the word ${_{2\mathbf u^\prime}z}\,{_{2\mathbf u^\prime}c}\,{_{2\mathbf u^\prime}x}$ is a factor of $\mathbf u^\prime$, and ${_{1\mathbf u^\prime}c}$ is not adjacent to ${_{1\mathbf u^\prime}x}$ and ${_{1\mathbf u^\prime}z}$ in $\mathbf u^\prime$.
Then $(\mathbf v_1^\prime)_c=\mathbf w_\nu$ by Lemma~\ref{lem:u_C} contradicting the fact that $\zeta\ne \nu$.
Thus, ${_{1\mathbf v_1}c}$ is not adjacent to ${_{1\mathbf v_1}x}$ in $\mathbf v_1$ in any case.

Further, since $xyx$ is an isoterm for the variety $\mathbb M(\mathscr W_n)\{\mathsf{Id}(\upsilon_{\mathscr W_n})\}$ by Lemma~\ref{lem:isotrems-for-M(W_n)Id(U_n)} and ${_{1\mathbf u}y}$ and ${_{1\mathbf u}c}$ do not lie in the same block of $\mathbf u$, the variables ${_{1\mathbf v_1}y}$ and ${_{1\mathbf v_1}c}$ cannot lie in the same block of $\mathbf v_1$.
Hence ${_{1\mathbf v_1}c}$ is not adjacent to ${_{1\mathbf v_1}y}$ in $\mathbf v_1$. So, if ${_{2\mathbf v_1}x}\,{_{2\mathbf v_1}c}\,{_{2\mathbf v_1}y}$ is a factor of $\mathbf v_1$, then we can apply the induction assumption, yielding that $\mathbf v_c=\mathbf w_\zeta$.
If ${_{2\mathbf v_1}z}\,{_{2\mathbf v_1}c}\,{_{2\mathbf v_1}x}$ is a factor of $\mathbf v_1$, then ${_{1\mathbf v_1}c}$ and ${_{1\mathbf v_1}x}$ must lie in the same block of $\mathbf v_1$ because $xzytxy$ is an isoterm for $\mathbb M(\mathscr W_n)\{\mathsf{Id}(\upsilon_{\mathscr W_n})\}$.
In this case, ${_{1\mathbf v_1}c}$ and ${_{1\mathbf v_1}z}$ lie in different blocks of $\mathbf v_1$ and so ${_{1\mathbf v_1}c}$ is not adjacent to ${_{1\mathbf v_1}z}$ in $\mathbf v_1$.
Therefore, we can apply the induction assumption again, yielding that $\mathbf v_c=\mathbf w_\zeta$ as required.
\end{proof}

\begin{lemma}
\label{lem:1x-1c-1y}
Let $n\ge2$ and $\mathbf u$ be a block-linear word such that $\mathbf u_c=\mathbf w_\zeta$ for some $\zeta\in S_2^n$ and $c\in\mathsf{mul}(\mathbf u)$ with $\mathsf{occ}_c(\mathbf u)=2$.
Assume that, for some $x,y\in \mathsf{mul}(\mathbf u_\zeta)$, the word ${_{1\mathbf u}x}\,{_{1\mathbf u}c}\,{_{1\mathbf u}y}$ is a factor of $\mathbf u$, while ${_{2\mathbf u}c}$ is not adjacent to ${_{2\mathbf u}x}$ and ${_{2\mathbf u}y}$ in $\mathbf u$.
If $\mathbf u\approx \mathbf v$ is an identity of $\mathbb M(\mathscr W_n)\{\mathsf{Id}(\upsilon_{\mathscr W_n})\}$, then $\mathbf v_c=\mathbf w_\zeta$.
\end{lemma}

\begin{proof}
Evidently, $x,y\in\{a,a_i,b,b_i,x_1^{(i)},x_2^{(i)}\mid 1\le i\le n\}$. 
If $({_{2\mathbf u}y_0})<({_{2\mathbf u}c})<({_{2\mathbf u}y_n})$, then the required claim follows from Lemma~\ref{lem:2x-2c-2y}(i). 
So, since the word $\mathbf u$ is block-linear, it remains to consider the case when one of the words ${_{2\mathbf u}c}\,{_{2\mathbf u}b}$, ${_{2\mathbf u}b}\,{_{2\mathbf u}c}$, ${_{2\mathbf u}c}\,{_{2\mathbf u}a}$, ${_{2\mathbf u}a}\,{_{2\mathbf u}c}$, ${_{2\mathbf u}c}\,{_{1\mathbf u}y_j}$ or ${_{1\mathbf u}y_j}\,{_{2\mathbf u}c}$ with $j\in\{0,1,\dots,n\}$ is a factor of $\mathbf u$.

In view of Proposition~\ref{prop:deduction}, there is some finite sequence $\mathbf u = \mathbf v_0, \mathbf v_1, \ldots, \mathbf v_m = \mathbf v$ of distinct words such that each identity $\mathbf v_i \approx \mathbf v_{i+1}$ either holds in $\mathbb M(\mathscr W_n)$ or is directly deducible from some identity in~$\mathsf{Id}(\upsilon_{\mathscr W_n})$.
We will use induction on $m$.

\smallskip

\textbf{Induction base:} $m=1$. 
If $\mathbf u=\mathbf v_0 \approx \mathbf v_1=\mathbf v$ holds in $\mathbb M(\mathscr W_n)$, then the required claim follows from Lemma~\ref{lem:M(W)}.
If $\mathbf u=\mathbf v_0 \approx \mathbf v_1=\mathbf v$ is directly deducible from some identity in~$\mathsf{Id}(\upsilon_{\mathscr W_n})$, then the condition of the lemma implies that every factor of length $>1$ of $\mathbf u$ has exactly one occurrence in $\mathbf u$ and $({_{1\mathbf u}a_1})<({_{1\mathbf u}c})<({_{1\mathbf u}b_n})$.
Then we can apply Lemma~\ref{lem:u_C}, yielding that $\mathbf v_c=\mathbf w_\zeta$.

\smallskip

\textbf{Induction step:} $m>1$. 
First, notice that, as in the induction base, $(\mathbf v_1)_c=\mathbf w_\zeta$ by Lemmas~\ref{lem:M(W)} and~\ref{lem:u_C}.
If either ${_{2\mathbf u}c}\,{_{1\mathbf u}y_j}$ or ${_{1\mathbf u}y_j}\,{_{2\mathbf u}c}$ is a factor of $\mathbf u$ for some $j\in\{0,1,\dots,n\}$, then $\mathbf u(c,s_j,t,x)=xcs_jctx$ and $\mathbf u(c,s_j,t,y)=cys_jcty$.
Since the word $xyzxty$ and so the word $yxzxty$ are isoterms for $\mathbb M(\mathscr W_n)\{\mathsf{Id}(\upsilon_{\mathscr W_n})\}$ by Lemma~\ref{lem:isotrems-for-M(W_n)Id(U_n)}, this implies that $\mathbf v_1(c,s_j,t,x)=xcs_jctx$ and $\mathbf v_1(c,s_j,t,y)=cys_jcty$.
If either ${_{2\mathbf u}c}\,{_{2\mathbf u}b}$ or ${_{2\mathbf u}b}\,{_{2\mathbf u}c}$ is a factor of $\mathbf u$, then $x\ne b$, $y\ne b$ and so $\mathbf u(c,s_0,t,x,y_0)=xcs_0y_0tcy_0x$ and $\mathbf u(c,s_0,t,y,y_0)=cys_0y_0tcy_0y$.
Since $xytzsxzy$ and so $yxtzsxzy$ are isoterms for $\mathbb M(\mathscr W_n)\{\mathsf{Id}(\upsilon_{\mathscr W_n})\}$ by Lemma \ref{lem:isotrems-for-M(W_n)Id(U_n)}, this implies that $\mathbf v_1(c,s_0,t,x,y_0)=xcs_0y_0tcy_0x$ and $\mathbf v_1(c,s_0,t,y,y_0)=cys_0y_0tcy_0y$.
By a similar argument we can show that if one of the words ${_{2\mathbf u}c}\,{_{2\mathbf u}a}$ or ${_{2\mathbf u}a}\,{_{2\mathbf u}c}$ is a factor of the word $\mathbf u$, then $\mathbf v_1(c,s_n,t,x,y_n)=xcs_ny_ntxy_nc$ and $\mathbf v_1(c,s_n,t,y,y_n)=cys_ny_ntyy_nc$.
Thus, we have proved that $\mathsf{occ}_c(\mathbf v_1)=2$, the word ${_{1\mathbf v_1}x}\,{_{1\mathbf v_1}c}\,{_{1\mathbf v_1}y}$ is a factor $\mathbf v_1$, while ${_{2\mathbf v_1}c}$ is not adjacent to ${_{2\mathbf v_1}x}$ and ${_{2\mathbf v_1}y}$ in $\mathbf v_1$.
So, we can apply the induction assumption, yielding that $\mathbf v_c=\mathbf w_\zeta$.
\end{proof}

\begin{corollary}
\label{cor:-ix-1h-iy}
Let $n\ge2$ and $\mathbf u$ be a word such that $\mathbf u_h=\mathbf w_\zeta$ for some $\zeta\in S_2^n$ and $h\in\mathsf{sim}(\mathbf u)$.
Assume that $h$ is adjacent to two different multiple variables of $\mathbf u$.
If $\mathbf u\approx \mathbf v$ is an identity of $\mathbb M(\mathscr W_n)\{\mathsf{Id}(\upsilon_{\mathscr W_n})\}$, then $\mathbf v_h=\mathbf w_\zeta$.
\end{corollary}

\begin{proof}
Obviously, there is $j\in\{0,1,\dots,n\}$ such that ${_{2\mathbf u}y_j}$ is not adjacent to ${_{1\mathbf u}h}$ in $\mathbf u$. 
Then Lemmas~\ref{lem:2x-2c-2y} and~\ref{lem:1x-1c-1y} imply that $(\varphi(\mathbf v))_h=(\varphi(\mathbf u))_h=\mathbf w_\zeta$, where $\varphi\colon \mathscr X \to \mathscr X^\ast$ is the substitution given by 
\[
\varphi(v):= 
\begin{cases} 
s_jh & \text{if $v=s_j$}, \\ 
v & \text{if $v\ne s_j$}.
\end{cases} 
\]
It remains to note that $\mathbf v_h=(\varphi(\mathbf v))_h$.
\end{proof}

\begin{lemma}
\label{lem:2x-2c1-2c2-2y}
Let $n\ge2$ and $\mathbf u$ be a word such that $\mathbf u_{\{c_1,c_2\}}=\mathbf w_\zeta$ for some $\zeta\in S_2^n$ and $c_1,c_2\in\mathsf{mul}(\mathbf u)$ with $\mathsf{occ}_{c_1}(\mathbf u)=\mathsf{occ}_{c_2}(\mathbf u)=2$.
Assume that, for some $x,y\in \mathsf{mul}(\mathbf w_\zeta)$, the word ${_{2\mathbf u}x}\,{_{2\mathbf u}c}_1\,{_{2\mathbf u}c}_2\,{_{2\mathbf u}y}$ is a factor of $\mathbf u$, while ${_{1\mathbf u}c_1}$ and ${_{1\mathbf u}c_2}$ lie in the same blocks as ${_{1\mathbf u}y}$ and ${_{1\mathbf u}x}$ in $\mathbf u$, respectively. 
If $\mathbf u\approx \mathbf v$ is an identity of $\mathbb M(\mathscr W_n)\{\mathsf{Id}(\upsilon_{\mathscr W_n})\}$, then $\mathbf v_{\{c_1,c_2\}}=\mathbf w_\zeta$.
\end{lemma}

\begin{proof}
In view of Proposition~\ref{prop:deduction}, there is some finite sequence $\mathbf u = \mathbf v_0, \mathbf v_1, \ldots, \mathbf v_m = \mathbf v$ of distinct words such that each identity $\mathbf v_i \approx \mathbf v_{i+1}$ either holds in $\mathbb M(\mathscr W_n)$ or is directly deducible from some identity in~$\mathsf{Id}(\upsilon_{\mathscr W_n})$.
We will use induction on $m$.

\smallskip

\textbf{Induction base:} $m=1$. 
If $\mathbf u=\mathbf v_0 \approx \mathbf v_1=\mathbf v$ holds in $\mathbb M(\mathscr W_n)$, then the required claim follows from Lemma~\ref{lem:M(W)}.
If $\mathbf u=\mathbf v_0 \approx \mathbf v_1=\mathbf v$ is directly deducible from some identity in~$\mathsf{Id}(\upsilon_{\mathscr W_n})$, then the condition of the lemma implies that every factor of length $>1$ of $\mathbf u$ has exactly one occurrence in $\mathbf u$ and $({_{2\mathbf u}b})<({_{2\mathbf u}c_1})<({_{2\mathbf u}a})$.
Then we can apply Lemma~\ref{lem:u_C}, yielding that $\mathbf v_{\{c_1,c_2\}}=\mathbf w_\zeta$.

\smallskip

\textbf{Induction step:} $m>1$. 
First, notice that, as in the induction base, $(\mathbf v_1)_{\{c_1,c_2\}}=\mathbf w_\zeta$ by Lemmas~\ref{lem:M(W)} and~\ref{lem:u_C}.
Since $xyx$ is an isoterm for the variety $\mathbb M(\mathscr W_n)\{\mathsf{Id}(\upsilon_{\mathscr W_n})\}$ by Lemma~\ref{lem:isotrems-for-M(W_n)Id(U_n)}, ${_{1\mathbf v_1}c_1}$ and ${_{1\mathbf v_1}c_2}$ lie in the same blocks as ${_{1\mathbf v_1}y}$ and ${_{1\mathbf v_1}x}$ in $\mathbf v_1$, respectively, and $\mathsf{occ}_{c_1}(\mathbf v_1)=\mathsf{occ}_{c_2}(\mathbf v_1)=2$. 
Further, ${_{1\mathbf u}y}$ and ${_{1\mathbf u}x}$ lie in different blocks of $\mathbf u$.
Hence, since $xzytxy$ is an isoterm for the variety $\mathbb M(\mathscr W_n)\{\mathsf{Id}(\upsilon_{\mathscr W_n})\}$ by Lemma~\ref{lem:isotrems-for-M(W_n)Id(U_n)}, the word ${_{2\mathbf v_1}x}\,{_{2\mathbf v_1}c}_1\,{_{2\mathbf v_1}c}_2\,{_{2\mathbf v_1}y}$ must be a factor of $\mathbf v_1$. 
Thus, we can apply the induction assumption, yielding that $\mathbf v_{\{c_1,c_2\}}=\mathbf w_\zeta$ as required.
\end{proof}

\begin{lemma}
\label{lem:1x-1c1-1c2-1y}
Let $n\ge2$ and $\mathbf u$ be a word such that $\mathbf u_{\{c_1,c_2\}}=\mathbf w_\zeta$ for some $\zeta\in S_2^n$ and $c_1,c_2\in\mathsf{mul}(\mathbf u)$ with $\mathsf{occ}_{c_1}(\mathbf u)=\mathsf{occ}_{c_2}(\mathbf u)=2$.
Assume that, for some $x,y\in \mathsf{mul}(\mathbf w_\zeta)$, the word ${_{1\mathbf u}x}\,{_{1\mathbf u}c}_1\,{_{1\mathbf u}c}_2\,{_{1\mathbf u}y}$ is a factor of $\mathbf u$, while ${_{2\mathbf u}c_1}$ and ${_{2\mathbf u}c_2}$ are adjacent to ${_{2\mathbf u}y}$ and ${_{2\mathbf u}x}$ in $\mathbf u$, respectively. 
If $\mathbf u\approx \mathbf v$ is an identity of $\mathbb M(\mathscr W_n)\{\mathsf{Id}(\upsilon_{\mathscr W_n})\}$, then $\mathbf v_{\{c_1,c_2\}}=\mathbf w_\zeta$.
\end{lemma}

\begin{proof}
In view of Proposition~\ref{prop:deduction}, there is some finite sequence $\mathbf u = \mathbf v_0, \mathbf v_1, \ldots, \mathbf v_m = \mathbf v$ of distinct words such that each identity $\mathbf v_i \approx \mathbf v_{i+1}$ either holds in $\mathbb M(\mathscr W_n)$ or is directly deducible from some identity in~$\mathsf{Id}(\upsilon_{\mathscr W_n})$.
We will use induction on $m$.

\smallskip

\textbf{Induction base:} $m=1$. 
If $\mathbf u=\mathbf v_0 \approx \mathbf v_1=\mathbf v$ holds in $\mathbb M(\mathscr W_n)$, then the required claim follows from Lemma~\ref{lem:M(W)}.
If $\mathbf u=\mathbf v_0 \approx \mathbf v_1=\mathbf v$ is directly deducible from some identity in~$\mathsf{Id}(\upsilon_{\mathscr W_n})$, then the condition of the lemma implies that every factor of length $>1$ of $\mathbf u$ has exactly one occurrence in $\mathbf u$ and $({_{1\mathbf u}a_1})<({_{1\mathbf u}c_1})<({_{1\mathbf u}b_n})$.
Then we can apply Lemma~\ref{lem:u_C}, yielding that $\mathbf v_{\{c_1,c_2\}}=\mathbf w_\zeta$.

\smallskip

\textbf{Induction step:} $m>1$. 
First, notice that, as in the induction base, $(\mathbf v_1)_{\{c_1,c_2\}}=\mathbf w_\zeta$ by Lemmas~\ref{lem:M(W)} and~\ref{lem:u_C}.
Since $xzytxy$ is an isoterm for the variety $\mathbb M(\mathscr W_n)\{\mathsf{Id}(\upsilon_{\mathscr W_n})\}$ by Lemma~\ref{lem:isotrems-for-M(W_n)Id(U_n)}, it is easy to show that $\mathsf{occ}_{c_1}(\mathbf v_1)=\mathsf{occ}_{c_2}(\mathbf v_1)=2$ and the variables ${_{2\mathbf v_1}c_1}$ and ${_{2\mathbf v_1}c_2}$ are adjacent to ${_{2\mathbf v_1}y}$ and ${_{2\mathbf v_1}x}$ in $\mathbf v_1$, respectively. 
Assume first that $\{x,y\}=\{x_1^{(k)},x_2^{(k)}\}$ for some $k\in\{1,2,\dots,n\}$.
Evidently, the words $\mathbf u_{\{x,c_1\}}$ and $\mathbf u_{\{c_2,y\}}$ coincide (up to renaming of variables) with $\mathbf w_\zeta$, while $\mathbf u_{\{x_1,x_2\}}$ coincides (up to renaming of variables) with $\mathbf w_{\bar{\zeta}}$, where $\bar{\zeta}:=(\zeta_1,\dots,\zeta_{k-1},\zeta_k^2,\zeta_{k+1},\dots,\zeta_n)$.
Hence if $\mathbf u \approx \mathbf v_1$ holds in $\mathbb M(\mathscr W_n)$, then ${_{1\mathbf v_1}x}\,{_{1\mathbf v_1}c}_1\,{_{1\mathbf v_1}c}_2\,{_{1\mathbf v_1}y}$ is a factor of $\mathbf v_1$ by Lemma~\ref{lem:M(W)}; if $\mathbf u \approx \mathbf v_1$ is directly deducible from some identity in~$\mathsf{Id}(\upsilon_{\mathscr W_n})$, then we apply Lemma~\ref{lem:u_C} three times, yielding that $({_{1\mathbf v_1}x})<({_{1\mathbf v_1}c}_1)<({_{1\mathbf v_1}c}_2)<({_{1\mathbf v_1}y})$ and so ${_{1\mathbf v_1}x}\,{_{1\mathbf v_1}c}_1\,{_{1\mathbf v_1}c}_2\,{_{1\mathbf v_1}y}$ is a factor of $\mathbf v_1$ again.
Assume now that $\{x,y\}\ne\{x_1^{(k)},x_2^{(k)}\}$ for all $k=1,2,\dots,n$.
In this case, there exists $j\in\{0,1,\dots,n\}$ such that $_{2\mathbf u}y_j$ lies between $_{2\mathbf u}x$ and $_{2\mathbf u}y$ in $\mathbf u$.
Then the words $\mathbf u(c_1,s_j,t,x,y_j)$, $\mathbf u(c_1,c_2,s_j,t,y_j)$ and $\mathbf u(c_2,s_j,t,y,y_j)$ coincide (up to renaming of variables) with either $xysztxzy$ or $yxsztxzy$. 
Since the latter two words are isoterms for $\mathbb M(\mathscr W_n)\{\mathsf{Id}(\upsilon_{\mathscr W_n})\}$ by Lemma~\ref{lem:isotrems-for-M(W_n)Id(U_n)}, we have $\mathbf u(c_1,s_j,t,x,y_j)=\mathbf v_1(c_1,s_j,t,x,y_j)$, $\mathbf u(c_1,c_2,s_j,t,y_j)=\mathbf v_1(c_1,c_2,s_j,t,y_j)$ and $\mathbf u(c_2,s_j,t,y,y_j)=\mathbf v_1(c_2,s_j,t,y,y_j)$.
It follows that $({_{1\mathbf v_1}x})<({_{1\mathbf v_1}c}_1)<({_{1\mathbf v_1}c}_2)<({_{1\mathbf v_1}y})$.
We see that ${_{1\mathbf v_1}x}\,{_{1\mathbf v_1}c}_1\,{_{1\mathbf v_1}c}_2\,{_{1\mathbf v_1}y}$ is a factor of $\mathbf v_1$ in any case.
Thus, we can apply the induction assumption, yielding that $\mathbf v_{\{c_1,c_2\}}=\mathbf w_\zeta$ as required.
\end{proof}

For any $n,m,k\ge1$ and $\rho\in S_{n+m+k}$, we define the words:
\begin{align*}
\mathbf c_{n,m,k}[\rho]&:=\biggl(\prod_{i=1}^n z_it_i\biggr)xyt\biggl(\prod_{i=n+1}^{n+m} z_it_i\biggr)x\biggl(\prod_{i=1}^{n+m+k} z_{i\rho}\biggr)y\biggl(\prod_{i=n+m+1}^{n+m+k} t_iz_i\biggr),\\[-3pt]
\mathbf c_{n,m,k}^\prime[\rho]&:=\biggl(\prod_{i=1}^n z_it_i\biggr)yxt\biggl(\prod_{i=n+1}^{n+m} z_it_i\biggr)x\biggl(\prod_{i=1}^{n+m+k} z_{i\rho}\biggr)y\biggl(\prod_{i=n+m+1}^{n+m+k} t_iz_i\biggr).
\end{align*}
Let $\mathbb O$ denote the variety defined by the first four identities in~\eqref{eq:five identities} together with all the identities of the form
\[
\mathbf c_{n,m,k}[\rho]\approx\mathbf c_{n,m,k}^\prime[\rho]
\]
with $n,m,k\ge1$ and $\rho\in S_{n+m+k}$.
An \textit{island} of a word $\mathbf w$ is a maximal factor of $\mathbf w$ that consists of only the second occurrences of variables whose first occurrences lie in the same block of $\mathbf w$. 
The next statement follows from the dual to Lemma~3.12 in~\cite{Gusev-Vernikov-21}.

\begin{lemma}
\label{lem:swapping}
If $\mathbf w:=\mathbf p\,{_{2\mathbf w}x}\,{_{2\mathbf w}y}\,\mathbf q$ and the variables ${_{2\mathbf w}x}$ and ${_{2\mathbf w}y}$ lie in the same island of $\mathbf w$, then $\mathbb O$ satisfies the identity $\mathbf w\approx\mathbf p\,yx\,\mathbf q$.\qed
\end{lemma}

If $xy$ is an isoterm for a variety $\mathbb V$, then it is easy to see that every identity of $\mathbb V$ is of the form
\begin{equation}
\label{eq:..t_iu_i..=..t_iv_i..}
\mathbf u_0\biggl(\prod_{i=1}^m t_i\mathbf u_i\biggr) \approx \mathbf v_0 \biggl(\prod_{i=1}^m t_i\mathbf v_i\biggr),
\end{equation}
where $\mathsf{sim}(\mathbf u)=\mathsf{sim}(\mathbf v) = \{t_1,t_2, \dots, t_m\}$ for some $m \ge 0$.
For each $i=0,1,\dots, m$, we say the blocks $\mathbf u_i$ and $\mathbf v_i$ are \textit{corresponding}.
An identity of the form~\eqref{eq:..t_iu_i..=..t_iv_i..} with $\mathsf{sim}(\mathbf u)=\mathsf{sim}(\mathbf v) = \{t_1,t_2, \dots, t_m\}$ is \textit{linear-balanced} if, for any $i=0,1,\dots, m$, the corresponding blocks $\mathbf u_i$ and $\mathbf v_i$ are linear words depending on the same variables. 
A linear-balanced identity $\mathbf u \approx \mathbf v$ is \textit{reduced} if all corresponding blocks of $\mathbf u$ and $\mathbf v$ are of the form $\mathbf a\mathbf c$ and $\mathbf b\mathbf c$, where $\mathbf a$ and $\mathbf b$ consist of the first occurrences of variables in $\mathbf u$ and $\mathbf v$, respectively, while $\mathbf c$ consists of the second occurrences of variables in both $\mathbf u$ and $\mathbf v$. 
Evidently, if $\mathbf u \approx \mathbf v$ is a reduced identity, then every variable occurs in both $\mathbf u$ and $\mathbf v$ at most twice.

\begin{lemma}
\label{lem:O}
Each variety in the interval $[\mathbb M(xzytxy),\mathbb O]$ may be defined within $\mathbb O$ by a set of reduced identities.
\end{lemma}

\begin{proof}
We need to show that an arbitrary identity $\mathbf u\approx \mathbf v$ of $\mathbb M(xzytxy)$ is equivalent within $\mathbb O$ to a reduced identity.
Let
\begin{align*}
&\mathscr A:=\{x\in\mathsf{mul}(\mathbf u)\mid \mathbf u(xyx)=xyx \text{ for some } y\in\mathsf{sim}(\mathbf u)\}\, \text{ and } \\ 
&\mathscr B:=\mathsf{mul}(\mathbf u)\setminus\mathscr A=\{x\in\mathsf{mul}(\mathbf u)\mid \mathsf{occ}_x(\mathbf u)>2 \text{ or } {_{1\mathbf u}x}\text{ and } {_{2\mathbf u}x} \text{ lie in the same block of } \mathbf u \}.
\end{align*}
Since the word $xyx$ is an isoterm for $\mathbb M(xzytxy)$, it is routine to show that
\begin{align*}
&\mathscr A=\{x\in\mathsf{mul}(\mathbf v)\mid \mathbf v(xyx)=xyx \text{ for some } y\in\mathsf{sim}(\mathbf v)\}\, \text{ and } \\ 
&\mathscr B=\mathsf{mul}(\mathbf v)\setminus\mathscr A=\{x\in\mathsf{mul}(\mathbf v)\mid \mathsf{occ}_x(\mathbf v)>2 \text{ or } {_{1\mathbf v}x}\text{ and } {_{2\mathbf v}x} \text{ lie in the same block of } \mathbf v \}.
\end{align*}
Let $\mathscr B=\{b_1,b_2,\dots,b_r\}$.
Arguments similar to those of the proof of Lemma~3.11 in~\cite{Gusev-Vernikov-21} imply that the identity
\[
\biggl(\prod_{i=1}^n z_it_i\biggr) x \biggl(\prod_{i=1}^{n+m} z_{i\rho}\biggr) x \biggl(\prod_{i=n+1}^{n+m} t_iz_i\biggr) \approx \biggl(\prod_{i=1}^n z_it_i\biggr) x^2 \biggl(\prod_{i=1}^{n+m} z_{i\rho}\biggr) \biggl(\prod_{i=n+1}^{n+m} t_iz_i\biggr) 
\] 
is satisfied by $\mathbb O$ for any $n,m\ge1$ and $\rho\in S_{n+m}$.
Then, by Lemma~4.5 in~\cite{Gusev-Vernikov-18}, the variety $\mathbb O$ satisfies the identities $\mathbf u\approx b_1^2\cdots b_r^2\mathbf u_{\mathscr B}$ and $\mathbf v\approx b_1^2\cdots b_r^2\mathbf v_{\mathscr B}$.
Hence $\mathbb O\{\mathbf u\approx \mathbf v\}=\mathbb O\{\mathbf u_{\mathscr B}\approx \mathbf v_{\mathscr B}\}$.
The identity $\mathbf u_{\mathscr B}\approx \mathbf v_{\mathscr B}$ is linear-balanced and every variable occurs in $\mathbf u_{\mathscr B}$ and $\mathbf v_{\mathscr B}$ at most twice.
Further, the fourth identity in~\eqref{eq:five identities} allows us to swap the first and the second occurrences of two multiple variables whenever these occurrences are adjacent to each other. 
In view of this fact, the variety $\mathbb O$ satisfies the identities $\mathbf u_{\mathscr B}\approx\mathbf w_1$ and $\mathbf v_{\mathscr B}\approx\mathbf w_2$ for some words $\mathbf w_1$ and $\mathbf w_2$ such that each block of $\mathbf w_1$ or $\mathbf w_2$ is a product of two words consisting of the first and the second occurrences of variables, respectively.
This means that $\mathbf w_i=\mathbf a_0^{(i)}\mathbf c_0^{(i)}t_1\mathbf a_1^{(i)}\mathbf c_1^{(i)}\cdots t_k\mathbf a_k^{(i)}\mathbf c_k^{(i)}$, where $\mathsf{sim}(\mathbf w_i)=\{t_1,t_2,\dots,t_k\}$ and, for any $j=0,1,\dots,k$, the word $\mathbf a_j^{(i)}$ [respectively, $\mathbf c_j^{(i)}$] consists of the first [second] occurrences of variables in $\mathbf w_i$.
Clearly, $\mathbf c_j^{(i)}$ can be represented as a product of some islands $\mathbf c_{j1}^{(i)}, \mathbf c_{j2}^{(i)},\dots, \mathbf c_{jr_j^{(i)}}^{(i)}$ of $\mathbf w_i$.
Since $\mathbf w_1\approx \mathbf w_2$ holds in $\mathbb M(xzytxy)$, it is easy to deduce from Lemma~\ref{lem:M(W)} that $r_j:=r_j^{(1)}=r_j^{(2)}$ and $\mathsf{con}(\mathbf c_{j\ell}^{(1)})=\mathsf{con}(\mathbf c_{j\ell}^{(2)})$ for any $j=0,1,\dots,k$ and $\ell=1,2,\dots,r_j$.
Now Lemma~\ref{lem:swapping} applies, yielding that $\mathbb O$ satisfies $\mathbf w_1\approx \mathbf w_1^\prime$, where $\mathbf w_1^\prime:=\mathbf a_0^{(1)}\mathbf c_0^{(2)}t_1\mathbf a_1^{(1)}\mathbf c_1^{(2)}\cdots t_k\mathbf a_k^{(1)}\mathbf c_k^{(2)}$.
Clearly, the identity $\mathbf w_1^\prime\approx \mathbf w_2$ is reduced and $\mathbb O\{\mathbf u\approx \mathbf v\}=\mathbb O\{\mathbf w_1^\prime\approx \mathbf w_2\}$ as required.
\end{proof}

We call an identity $\mathbf w\approx\mathbf w^\prime$ 1-\textit{invertible} if $\mathbf w=\mathbf a\, xy\,\mathbf b$ and $\mathbf w^\prime=\mathbf a\, yx\,\mathbf b$ for some $\mathbf a,\mathbf b\in\mathscr X^\ast$ and $x,y\in\mathsf{con}(\mathbf a\mathbf b)$. 
Let $k>1$. 
An identity $\mathbf w\approx\mathbf w^\prime$ is called $k$-\textit{invertible} if there is a sequence of words $\mathbf w=\mathbf w_0,\mathbf w_1,\dots,\mathbf w_k=\mathbf w^\prime$ such that the identity $\mathbf w_i\approx\mathbf w_{i+1}$ is 1-invertible for each $i=0,1,\dots,k-1$ and $k$ is the least number with such a property. 
For convenience, we will call the trivial identity 0-\textit{invertible}.

\smallskip

For the rest of this section, the mapping $\Phi\colon \mathfrak{Eq}(\mathscr W_n) \to [\mathbb M(\mathscr W_n)\{\mathsf{Id}(\upsilon_{\mathscr W_n})\},\mathbb M(\mathscr W_n)]$ given by 
\[
\Phi(\pi) := \mathbb M(\mathscr W_n)\{\mathsf{Id}(\pi)\}
\] 
is shown to be an anti-isomorphism.
The proof of Proposition~\ref{prop:mapping} is thus complete.

\subsection*{The mapping $\Phi$ is injective}

Suppose that $\Phi(\pi) = \Phi(\rho)$ for some $\pi,\rho \in\mathfrak{Eq}(\mathscr W_n)$, so that $\mathbb M(\mathscr W_n)\{\mathsf{Id}(\pi)\} = \mathbb M(\mathscr W_n)\{\mathsf{Id}(\rho)\}$.
If $(\mathbf u,\mathbf v) \in \rho$, then the variety $\mathbb M(\mathscr W_n)\{\mathsf{Id}(\pi)\}$ satisfies the identity $\mathbf u \approx \mathbf v$, whence $(\mathbf u,\mathbf v) \in \pi$ by Corollary~\ref{cor:FIC(M(W_n)Id(pi))-class}.
Therefore the inclusion $\rho \subseteq \pi$ holds; the reverse inclusion $\rho \supseteq \pi$ holds by a symmetrical argument, thus $\pi = \rho$.

\subsection*{The mapping $\Phi$ is surjective}

It suffices to show that for any variety $\mathbb V$ from the interval $[\mathbb M(\mathscr W_n)\{\mathsf{Id}(\upsilon_{\mathscr W_n})\},\mathbb M(\mathscr W_n)]$, there exists some $\pi \in \mathfrak{Eq}(\mathscr W_n)$ such that $\Phi(\pi) = \mathbb V$.
Since $\Phi(\varepsilon_{\mathscr W_n}) = \mathbb M(\mathscr W_n)\{ \mathsf{Id}(\varepsilon_{\mathscr W_n}) \} = \mathbb M(\mathscr W_n)$, where $\varepsilon_{\mathscr W_n}$ is the equality relation on $\mathscr W_n$, suppose that $\mathbb V \neq \mathbb M(\mathscr W_n)$.
Then there exists a nontrivial set~$\Sigma$ of identities such that $\mathbb V = \mathbb M(\mathscr W_n)\{\Sigma\}$; by the inclusions $\mathbb M(xzytxy)\subseteq\mathbb M(\mathscr W_n)\{\mathsf{Id}(\upsilon_{\mathscr W_n})\}\subseteq\mathbb M(\mathscr W_n)\subseteq\mathbb O$ and Lemma~\ref{lem:O}, the identities in~$\Sigma$ can be chosen to be reduced.
It is shown below that any identity $\mathbf u \approx \mathbf v$ in~$\Sigma$ is equivalent within $\mathbb M(\mathscr W_n)$ to a subset of~$\mathsf{Id}(\upsilon_{\mathscr W_n})$.
By Lemma~2.2 in~\cite{Gusev-Lee-20}, there exists some $\pi \in \mathfrak{Eq}(\mathscr W_n)$ such that $\mathbb V= \mathbb M(\mathscr W_n)\{\mathsf{Id}(\pi)\}$, so that $\Phi(\pi) = \mathbb V$ as required. 

Since the identity $\mathbf u \approx \mathbf v$ is reduced (and so linear-balanced), this identity is $r$-invertible for some $r\ge0$. 
We will use induction by $r$.

\smallskip

\textbf{Induction base}: $r=0$.
Then $\mathbf u=\mathbf v$, whence $\mathbb M(\mathscr W_n)\{\mathbf u\approx\mathbf v\}=\mathbb M(\mathscr W_n)\{\emptyset\}$.

\smallskip

\textbf{Induction step}: $r>0$.
If $\mathbf u\approx\mathbf v$ holds in $\mathbb M(\mathscr W_n)$, then $\mathbb M(\mathscr W_n)\{\mathbf u\approx\mathbf v\}=\mathbb M(\mathscr W_n)\{\emptyset\}$.
So, we may further assume that $\mathbf u \approx \mathbf v$ is violated by $\mathbb M(\mathscr W_n)$.
Then there is a substitution $\psi\colon \mathscr X \to \mathbb M(\mathscr W_n)$ such that $\psi(\mathbf u)\ne \psi(\mathbf v)$ in $\mathbb M(\mathscr W_n)$.
This is only possible when $\psi(\mathbf u)$ or $\psi(\mathbf v)$, say $\psi(\mathbf u)$, is a non-empty factor of some word $\mathbf w_\xi$ in $\mathscr W_n$. 
According to Lemma~\ref{lem:three-isotrems}(iii), every proper factor of $\mathbf w_\xi$ is an isoterm for $\mathbb M(\mathscr W_n)\{\mathsf{Id}(\upsilon_{\mathscr W_n})\}$.
Hence $\mathbf w_\xi = \psi(\mathbf u)$.
Clearly, $\psi(\mathbf v)$ represents a non-empty word, which does not equal to $\psi(\mathbf u)$.
In view of Lemma~\ref{lem:FIC(M(W_n)Id(U_n))-class}, $\psi(\mathbf v)=\mathbf w_\eta$ for some $\eta\in S_2^n\setminus\{\xi\}$.
Let $\mathscr V:=\{x\in \mathscr X\mid \psi(x)\ne1\}$.
Clearly, $\mathbf w_\xi= \psi(\mathbf u)=\psi(\mathbf u(\mathscr V))$ and $\mathbf w_\eta= \psi(\mathbf v)=\psi(\mathbf v(\mathscr V))$.

Notice that every factor of length $>1$ of $\mathbf w_\xi$ has exactly one occurrence in $\mathbf w_\xi$.
It follows that $\psi(v)$ is a variable for any $v\in\mathscr V\cap\mathsf{mul}(\mathbf u)$.
Let us now consider an arbitrary variable $c\in \mathscr V\cap\mathsf{sim}(\mathbf u)$.
If $\psi(c)$ is not a variable, then $\psi_c(\mathbf u(\mathscr V))$ is obtained from $\mathbf w_\xi$ by replacing some factor of length $>1$ with the variable $c$, where $\psi_c\colon \mathscr X\to \mathscr X^\ast$ is the substitution defined by
\[
\psi_c(v):=
\begin{cases}
\psi(v)&\text{if }v\ne c,\\
c&\text{if }v= c.
\end{cases}
\]
According to Lemma~\ref{lem:three-isotrems}(ii), the word $\psi_c(\mathbf u(\mathscr V))$ is an isoterm for $\mathbb M(\mathscr W_n)\{\mathsf{Id}(\upsilon_{\mathscr W_n})\}$ contradicting the fact that $\psi(\mathbf u(\mathscr V))\approx \psi(\mathbf v(\mathscr V))$ is a nontrivial identity of $\mathbb M(\mathscr W_n)\{\mathsf{Id}(\upsilon_{\mathscr W_n})\}$.
Therefore, $\psi(c)$ is a variable.
Further, if $\psi(c)\in \mathsf{mul}(\mathbf w_\xi)$, then $\psi_c(\mathbf u(\mathscr V))$ is obtained from $\mathbf w_\xi$ by replacing some occurrence of the multiple variable $\psi(c)$ with the variable $c$.
In view of Lemma~\ref{lem:three-isotrems}(i), the word $\psi_c(\mathbf u(\mathscr V))$ is an isoterm for $\mathbb M(\mathscr W_n)\{\mathsf{Id}(\upsilon_{\mathscr W_n})\}$ contradicting the fact that $\psi(\mathbf u(\mathscr V))\approx \psi(\mathbf v(\mathscr V))$ is a nontrivial identity of $\mathbb M(\mathscr W_n)\{\mathsf{Id}(\upsilon_{\mathscr W_n})\}$ again.
Therefore, $\psi(c)\in\mathsf{sim}(\mathbf w_\xi)$.
Since the identity $\mathbf w_\xi\approx \mathbf w_\eta$ is reduced, $\psi(\mathbf v_1)$ and $\psi(\mathbf v_2)$ cannot coincide with each other for distinct $\mathbf v_1,\mathbf v_2\in\mathscr V$.
Therefore, $\mathbf u(\mathscr V)$ and $\mathbf v(\mathscr V)$ coincide (up to renaming of variables) with $\mathbf w_\xi$ and $\mathbf w_\eta$, respectively.
We may assume without any loss that $\mathbf u(\mathscr V)=\mathbf w_\xi$ and $\mathbf v(\mathscr V)=\mathbf w_\eta$.

Let $\mathscr V^\prime:=\mathscr V\cap\mathsf{mul}(\mathbf u)$. 
For any $c\in \mathscr V^\prime$, let $\check{\mathbf c}$ denote the island of $\mathbf u$ containing ${_{2\mathbf u}c}$. 
Consider arbitrary $x,y\in\mathscr V^\prime$ such that ${_{2\mathbf w_\xi}x}\, {_{2\mathbf w_\xi}y}$ is a factor of $\mathbf w_\xi$.
Clearly, ${_{1\mathbf u}x}$ and ${_{1\mathbf u}y}$ lie in different blocks of $\mathbf u$, whence $\mathsf{con}(\check{\mathbf x})\cap\mathsf{con}(\check{\mathbf y})=\emptyset$.
Denote by $\mathbf c$ the factor of $\mathbf u$ lying between the factors $\check{\mathbf x}$ and $\check{\mathbf y}$.
Assume that $\mathbf c$ is non-empty. 
Corollary~\ref{cor:-ix-1h-iy} together with the fact that $\mathbf u(\mathscr V)=\mathbf w_\xi$ and $\mathbf v(\mathscr V)=\mathbf w_\eta$ imply that $\mathsf{con}(\mathbf c)\subseteq\mathsf{mul}(\mathbf u)$.
Since the identity $\mathbf u\approx \mathbf v$ is reduced, this implies that $\mathbf c$ consists of the second occurrences of variables in $\mathbf u$.
Let $c_1$ denote the first variable of $\mathbf c$.
The variables ${_{1\mathbf u}c_1}$ and ${_{1\mathbf u}x}$ do not lie in the same block  of $\mathbf u$ because ${_{1\mathbf u}c_1}$ belongs to the island $\check{\mathbf x}$ otherwise.
Therefore, there is $h\in\mathsf{sim}(\mathbf u)$ such that ${_{1\mathbf u}h}$ lies between ${_{1\mathbf u}x}$ and ${_{1\mathbf u}c_1}$ in $\mathbf u$.
If the variables ${_{1\mathbf u_1}c_1}$ and ${_{1\mathbf u_1}x}$ lie in the same block of $\mathbf u_1:=\mathbf u(\mathscr V\cup\{c_1\})$, then taking into account Corollary~\ref{cor:-ix-1h-iy} and the fact that $\mathbf u(\mathscr V)=\mathbf w_\xi$ and $\mathbf v(\mathscr V)=\mathbf w_\eta$, we conclude that $h\notin \mathscr V$ and, in the word $\mathbf u_2:=\mathbf u(\mathscr V\cup\{h,c_1\})$, the variable ${_{1\mathbf u_2}h}$ does not lie between ${_{1\mathbf u_2}a_1}$ and ${_{1\mathbf u_2}b_n}$.
Hence ${_{1\mathbf u_2}c_1}$ forms a block of $\mathbf u_2$ contradicting Lemma~\ref{lem:u_{c,h}}.
Therefore, ${_{1\mathbf u_1}c_1}$ and ${_{1\mathbf u_1}x}$ lie in different blocks of $\mathbf u_1$.
Then ${_{1\mathbf u_1}y}$ and ${_{1\mathbf u_1}c_1}$ lie in the same block of $\mathbf u_1$ by Lemma~\ref{lem:2x-2c-2y}.
By similar arguments we can show that if $c_2$ is the last variable of $\mathbf c$, then ${_{1\mathbf u_3}x}$ and ${_{1\mathbf u_3}c_2}$ lie in the same block of $\mathbf u_3:=\mathbf u(\mathscr V\cup\{c_2\})$ (and so $c_1\ne c_2$). 
This implies that the word $\mathbf u_4:=\mathbf u(\mathscr V\cup\{c_1,c_2\})$ contains the factor ${_{2\mathbf u_4}x}\,{_{2\mathbf u_4}c}_1\,{_{2\mathbf u_4}c}_2\,{_{2\mathbf u_4}y}$, while ${_{1\mathbf u_4}c_1}$ and ${_{1\mathbf u_4}c_2}$ lie in the same blocks as ${_{1\mathbf u_4}y}$ and ${_{1\mathbf u_4}x}$ in $\mathbf u_4$, respectively.
This contradicts Lemma~\ref{lem:2x-2c1-2c2-2y} because $\mathbf v(\mathscr V)=\mathbf w_\xi$ and $\mathbf v(\mathscr V)=\mathbf w_\eta$.
Therefore, the word $\mathbf c$ must be empty.
Since the variables $x$ and $y$ are arbitrary, we have proved that the word
\[
\mathbf r:=\check{\mathbf b}\,\check{\mathbf  y}_0\,\biggl(\prod_{i=1}^n \check{\mathbf x}_1^{(i)}\,\check{\mathbf z}_i \check{\mathbf a}_i\check{\mathbf z}_i^{\prime}\check{\mathbf b}_i\check{\mathbf z}_i^{\prime\prime} \,\check{\mathbf x}_2^{(i)}\, \check{\mathbf  y}_i\biggr)\,\check{\mathbf a}
\]
forms a factor of $\mathbf u$.

Further, for any $c\in \mathscr V^\prime\setminus\{a,b\}$, let $\hat{\mathbf c}$ denote the minimal factor of $\mathbf u$ containing all first occurrences of variables in $\mathsf{con}(\check{\mathbf c})$. 
Consider an arbitrary variable $x\in\mathscr V^\prime\setminus\{a,b\}$.
Let $d$ denote the last variable of $\hat{\mathbf x}$.
By the definition of $\hat{\mathbf x}$, we have $d\in\mathsf{con}(\check{\mathbf x})$.
Consider an arbitrary variable $e\in\mathsf{con}(\hat{\mathbf x})\setminus\mathsf{con}(\check{\mathbf x})$ such that some occurrence of $e$ lies between ${_{1\mathbf u}}x$ and ${_{1\mathbf u}}d$ in $\mathbf u$. 
Since the identity $\mathbf u\approx \mathbf v$ is reduced, this occurrence of $e$ must be the first one in $\mathbf u$.
By the definition of the island $\check{\mathbf x}$, the variable $e$ is multiple in $\mathbf u$.
Denote by $y$ the variable different from $x$ that is adjacent to the second occurrence  of $d$ in the word $\mathbf u_5:=\mathbf u(\mathscr V\cup\{d\})$.
Then, since $x\in \mathscr V^\prime\setminus\{a,b\}$ and the word $\mathbf u$ is block-linear, we have $y\in\mathscr V^\prime$ and either ${_{2\mathbf u_5}}x\,{_{2\mathbf u_5}}d\,{_{2\mathbf u_5}}y$ or ${_{2\mathbf u_5}}y\,{_{2\mathbf u_5}}d\,{_{2\mathbf u_5}}x$ is a factor of $\mathbf u_5$.
Further, since ${_{1\mathbf u_5}}x$ and ${_{1\mathbf u_5}}y$ lie in different blocks of $\mathbf u_5$, the variable ${_{1\mathbf u_5}}d$ is not adjacent to the variable ${_{1\mathbf u_5}}y$ in $\mathbf u_5$ as well. 
If $e\in \mathscr V$, then ${_{1\mathbf u_5}}d$ is not adjacent to ${_{1\mathbf u_5}}x$ in $\mathbf u_5$ because $({_{1\mathbf u}}x)<({_{1\mathbf u}}e)<({_{1\mathbf u}}d)$.
Then $\mathbf v(\mathscr V)=\mathbf w_\xi$ by Lemma~\ref{lem:2x-2c-2y}(i) contradicting $\mathbf v(\mathscr V)=\mathbf w_\eta$.
Therefore, $e\notin \mathscr V$.
Since the variable $e$ is arbitrary, we have proved that there are no variables in $\mathscr V$ lying between ${_{1\mathbf u}}x$ and ${_{1\mathbf u}}d$ in $\mathbf u$.

Suppose that $x=b_n$. 
In this case, ${_{2\mathbf u}}d$ lies between ${_{2\mathbf u}}z_n^\prime$ and ${_{2\mathbf u}}z_n^{\prime\prime}$ in $\mathbf u$, while ${_{2\mathbf u}}e$ does not.
It follows from the fact that $\mathbf u(\mathscr V)=\mathbf w_\xi$ and $\mathbf v(\mathscr V)=\mathbf w_\eta$ and Lemma~\ref{lem:2x-2c-2y}(i) that either $({_{2\mathbf u}}e)<({_{2\mathbf u}}y_0)$ or $({_{2\mathbf u}}y_n)<({_{2\mathbf u}}e)$.
It is easy to see that the word $\mathbf u_6:=\mathbf u((\mathscr V\setminus\{x\})\cup\{d\})$ coincides (up to renaming variables) with $\mathbf u(\mathscr V)=\mathbf w_\xi$.
Since the word $xzytxy$ and so the word $xzytyx$ are isoterms for $\mathbb M(\mathscr W_n)\{\mathsf{Id}(\upsilon_{\mathscr W_n})\}$ by Lemma~\ref{lem:isotrems-for-M(W_n)Id(U_n)}, we have $({_{2\mathbf v}}z_n^\prime)<({_{2\mathbf v}}d)<({_{2\mathbf v}}z_n^{\prime\prime})$.
Further, since $\mathbf u(b_{n-1},d,s_{n-1},t,y_{n-1})=b_{n-1}ds_{n-1}y_{n-1}tb_{n-1}y_{n-1}d$, we can apply Lemma~\ref{lem:isotrems-for-M(W_n)Id(U_n)} again, yielding that $\mathbf v(b_{n-1},d,s_{n-1},t,y_{n-1})=b_{n-1}ds_{n-1}y_{n-1}tb_{n-1}y_{n-1}d$.
Hence $({_{1\mathbf v}}b_{n-1})<({_{1\mathbf v}}d)$.
It follows that the word $\mathbf v((\mathscr V\setminus\{x\})\cup\{d\})$ coincides (up to renaming variables) with $\mathbf v(\mathscr V)=\mathbf w_\eta$.
However, since ${_{1\mathbf u_7}}b_{n-1}\,{_{1\mathbf u_7}}e\,{_{1\mathbf u_7}}d$ is a factor of $\mathbf u_7:=\mathbf u((\mathscr V\setminus\{x\})\cup\{d,e\})$ and either $({_{2\mathbf u_7}}e)<({_{2\mathbf u_7}}y_0)$ or $({_{2\mathbf u_7}}y_n)<({_{2\mathbf u_7}}e)$, Lemma~\ref{lem:1x-1c-1y} implies that $\mathbf v((\mathscr V\setminus\{x\})\cup\{d\})$ must coincide (up to renaming variables) with $\mathbf w_\xi$, a contradiction.

Suppose now that $x\ne b_n$.
Denote by $z$ the variable that directly follows ${_{1\mathbf w_\xi}}x$ in $\mathbf w_\xi$.
In view of the above, $({_{1\mathbf u}}d)<({_{1\mathbf u}}z)$, whence ${_{1\mathbf u_8}}x\,{_{1\mathbf u_8}}e\,{_{1\mathbf u_8}}z$ is a factor of the word $\mathbf u_8:=\mathbf u(\mathscr V\cup\{e\})$.
Clearly, ${_{2\mathbf u_8}}e$ is not adjacent to ${_{2\mathbf u_8}}x$ in $\mathbf u_8$.
Then Lemma~\ref{lem:1x-1c-1y} and the fact that $\mathbf u(\mathscr V)=\mathbf w_\xi$ and $\mathbf v(\mathscr V)=\mathbf w_\eta$ imply that ${_{2\mathbf u_8}}e$ is adjacent to ${_{2\mathbf u_8}}z$ in $\mathbf u_8$ contradicting Lemma~\ref{lem:1x-1c1-1c2-1y}.

Thus, we have proved that if $e$ is a variable lying between ${_{1\mathbf u}}x$ and ${_{1\mathbf u}}d$ in $\mathbf u$, then $e\in \mathsf{con}(\check{\mathbf x})$.
By similar arguments we can show that if $d^\prime$ is the first variable of $\hat{\mathbf x}$, then every variable lying between ${_{1\mathbf u}}d^\prime$ and ${_{1\mathbf u}}x$ in $\mathbf u$ must belong to $\mathsf{con}(\check{\mathbf x})$.
Therefore, we have proved that $\mathsf{con}(\hat{\mathbf c})=\mathsf{con}(\check{\mathbf c})$ for any $c\in\mathscr V^\prime\setminus\{a,b\}$.

Now consider arbitrary $x,y\in\mathscr V^\prime\setminus\{a,b\}$ such that ${_{1\mathbf w_\xi}x}\, {_{1\mathbf w_\xi}y}$ is a factor of $\mathbf w_\xi$.
Denote by $\mathbf c$ the factor of $\mathbf u$ lying between the factors $\hat{\mathbf x}$ and $\hat{\mathbf y}$.
Assume that $\mathbf c$ is non-empty. 
Then we can take $c\in\mathsf{con}(\mathbf c)$.
Corollary~\ref{cor:-ix-1h-iy} and the fact that $\mathbf u(\mathscr V)=\mathbf w_\xi$ and $\mathbf v(\mathscr V)=\mathbf w_\eta$ imply that $c\in\mathsf{mul}(\mathbf u)$.
Clearly, the variable ${_{2\mathbf u}c}$ does not occur in the islands $\check{\mathbf x}$ and $\check{\mathbf y}$ of $\mathbf u$.
It follows that the second occurrence of $c$ is not adjacent to the second occurrences of $x$ and $y$ in $\mathbf u(\mathscr V\cup\{c\})$ contradicting Lemma~\ref{lem:1x-1c-1y}.
Therefore, the word $\mathbf c$ must be empty.
Since the variables $x$ and $y$ are arbitrary, we have proved that the word 
\[
\mathbf h:=\biggl(\prod_{i=1}^n\hat{\mathbf a}_i\biggr)\,\hat{\mathbf a}\, \biggl(\prod_{i=1}^n\hat{\mathbf x}_{1\xi_i}^{(i)}\hat{\mathbf x}_{2\xi_i}^{(i)}\biggr)\,\hat{\mathbf b}\,\biggl(\prod_{i=1}^n\hat{\mathbf b}_i\biggr)
\]
forms a factor of $\mathbf u$, where $\hat{\mathbf a}$ [respectively, $\hat{\mathbf b}$] denote the factor of $\mathbf u$ lying between the factors $\hat{\mathbf a}_n$ and $\hat{\mathbf x}_{1\xi_1}^{(1)}$  [respectively, $\hat{\mathbf x}_{2\xi_n}^{(2)}$ and $\hat{\mathbf b}_1$]. 
Evidently, $a\in\mathsf{con}(\hat{\mathbf a})$ and $b\in\mathsf{con}(\hat{\mathbf b})$.

Consider an arbitrary variable $c\in\mathsf{con}(\hat{\mathbf a})\setminus \mathsf{con}(\check{\mathbf a})$.
It follows from Corollary~\ref{cor:-ix-1h-iy} that $c\in\mathsf{mul}(\mathbf u)$.
Further, since $\mathsf{con}(\hat{\mathbf a}_n)=\mathsf{con}(\check{\mathbf a}_n)$ and $\mathsf{con}(\hat{\mathbf x}_{1\xi_1}^{(1)})=\mathsf{con}(\check{\mathbf x}_{1\xi_1}^{(1)})$, Lemma~\ref{lem:1x-1c-1y} implies that the second occurrence of $c$ is adjacent to the second occurrence of $a$ in $\mathbf u(\mathscr V\cup\{c\})$.
By the definition of the island $\check{\mathbf a}$, either ${_{2\mathbf u}c}$ and ${_{2\mathbf u}a}$ lie in different blocks of $\mathbf u$ or ${_{2\mathbf u}c}$ and ${_{2\mathbf u}a}$ lie in the same block of $\mathbf u$ but in different islands of this block.
If ${_{2\mathbf u}c}$ and ${_{2\mathbf u}a}$ lie in different blocks of $\mathbf u$, then there is $h\in\mathsf{sim}(\mathbf u)$ such that $({_{2\mathbf u}a})<({_{1\mathbf u}h})<({_{2\mathbf u}c})$ contradicting Lemma~\ref{lem:u_{c,h}} because the second occurrence of $c$ in the word $\mathbf u(\mathscr V\cup\{c,h\})$ forms a block in this word.
If ${_{2\mathbf u}c}$ and ${_{2\mathbf u}a}$ lie in the same block of $\mathbf u$ but in different islands of this block, then there is $c_1\in \mathsf{mul}(\mathbf u)$ such that $({_{2\mathbf u}a})<({_{2\mathbf u}c_1})<({_{2\mathbf u}c})$ and ${_{1\mathbf u}c_1}$ do not lie in the block of $\mathbf u$ containing ${_{1\mathbf u}a}$ and ${_{1\mathbf u}c}$.
Since the identity $\mathbf u((\mathscr V\setminus\{a\})\cup\{c\})\approx \mathbf v((\mathscr V\setminus\{a\})\cup\{c\})$ coincides (up to renaming variables) with $\mathbf w_\xi\approx \mathbf w_\eta$ and ${_{2\mathbf u_9}y_n}\,{_{2\mathbf u_9}c_1}\,{_{2\mathbf u_9}c}$ is a factor of the word $\mathbf u_9:=\mathbf u((\mathscr V\setminus\{a\})\cup\{c,c_1\})$, Lemma~\ref{lem:2x-2c-2y}(iii) implies that ${_{1\mathbf u_9}c_1}$ is adjacent to ${_{1\mathbf u_9}y_n}$ in $\mathbf u_9$ contradicting Lemma~\ref{lem:2x-2c1-2c2-2y}.
Therefore, $\mathsf{con}(\hat{\mathbf a})\subseteq \mathsf{con}(\check{\mathbf a})$.
By similar arguments we can show that $\mathsf{con}(\hat{\mathbf b})\subseteq \mathsf{con}(\check{\mathbf b})$.

In view of the above, there are words $\mathbf t$, $\mathbf t_i$, $\mathbf t_i^\prime$, $\mathbf t_i^{\prime\prime}$ and $\mathbf s_i$ such that the word $\mathbf p\mathbf h\mathbf q\mathbf r$ is a factor of $\mathbf u$, where
\[
\mathbf p:=\biggl(\prod_{i=1}^n \hat{\mathbf z}_i\mathbf t_i\biggr)\biggl(\prod_{i=1}^n\hat{\mathbf z}_i^{\prime}\mathbf t_i^{\prime}\biggr)\biggl(\prod_{i=1}^n\hat{\mathbf z}_i^{\prime\prime}\mathbf t_i^{\prime\prime}\biggr),\ \  
\mathbf q:=\biggl(\prod_{i=0}^n \mathbf s_i\hat{\mathbf y}_i\biggr)\mathbf t.
\]
By similar arguments one can show that the word $\mathbf v$ contains a factor
\[
\bar{\mathbf p}\cdot \biggl(\prod_{i=1}^n \hat{\bar{\mathbf a}}_i\biggr)\, \biggl(\prod_{i=1}^n \hat{\bar{\mathbf x}}_{1\eta_i}^{(i)} \hat{\bar{\mathbf x}}_{2\eta_i}^{(i)}\biggr)\,\biggl(\prod_{i=1}^n \hat{\bar{\mathbf b}}_i\biggr) \cdot\bar{\mathbf q}\bar{\mathbf r}
\]
with 
\[
\bar{\mathbf p}:=\biggl(\prod_{i=1}^n \hat{\bar{\mathbf z}}_i\bar{\mathbf t}_i\biggr)\biggl(\prod_{i=1}^n\hat{\bar{\mathbf z}}_i^{\prime}\bar{\mathbf t}_i^{\prime}\biggr)\biggl(\prod_{i=1}^n\hat{\bar{\mathbf z}}_i^{\prime\prime}\bar{\mathbf t}_i^{\prime\prime}\biggr),
\bar{\mathbf q}:=\biggl(\prod_{i=0}^n \bar{\mathbf s}_i\hat{\bar{\mathbf y}}_i\biggr)\bar{\mathbf t},\,
\bar{\mathbf r}:=\check{\bar{\mathbf b}}\check{\bar{\mathbf y}}_0\biggl(\prod_{i=1}^n \check{\bar{\mathbf x}}_1^{(i)}\,\check{\bar{\mathbf z}}_i \check{\bar{\mathbf a}}_i\check{\bar{\mathbf z}}_i^{\prime}\check{\bar{\mathbf b}}_i\check{\bar{\mathbf z}}_i^{\prime\prime} \,\check{\bar{\mathbf x}}_2^{(i)}\, \check{\bar{\mathbf y}}_i\biggr)\check{\bar{\mathbf a}}
\]
such that 
\begin{itemize}
\item for any $c\in\mathscr V^\prime$, the word $\check{\bar{\mathbf c}}$ is the island of $\mathbf v$ containing ${_{2\mathbf v}c}$;
\item $\mathsf{con}(\hat{\bar{\mathbf c}})=\mathsf{con}(\check{\bar{\mathbf c}})$ for any $c\in\mathscr V^\prime\setminus\{a,b\}$;
\item $c\in\mathsf{con}(\hat{\bar{\mathbf c}})\subseteq\mathsf{con}(\check{\bar{\mathbf c}})$ and for any $c\in\{a,b\}$.
\end{itemize}
Since the identity $\mathbf u\approx \mathbf v$ is reduced, $\mathsf{con}(\check{\mathbf c})=\mathsf{con}(\check{\bar{\mathbf c}})$ for any $c\in\mathscr V^\prime$.
Hence $\mathsf{con}(\hat{\mathbf c})=\mathsf{con}(\hat{\bar{\mathbf c}})$ for any $c\in\mathscr V^\prime\setminus\{a,b\}$.
Further, one can deduce from Lemma~\ref{lem:FIC(M(W_n)Id(U_n))-class} that $\mathsf{con}(\hat{\mathbf c})=\mathsf{con}(\hat{\bar{\mathbf c}})$ for any $c\in\{a,b\}$.

Clearly, $\mathbf u=\mathbf a\cdot\mathbf p\mathbf h\mathbf q\mathbf r\cdot\mathbf b$ for some $\mathbf a,\mathbf b\in\mathscr X^\ast$.
Define $\mathbf w :=\mathbf a\cdot \mathbf p\tilde{\mathbf h}\mathbf q\mathbf r\cdot \mathbf b$, where
\[
\tilde{\mathbf h}: =\biggl(\prod_{i=1}^n \hat{\mathbf a}_i\biggr)\,\hat{\mathbf a}\, \biggl(\prod_{i=1}^n \hat{\mathbf x}_{1\eta_i}^{(i)} \hat{\mathbf x}_{2\eta_i}^{(i)}\biggr)\,\hat{\mathbf b}\,\biggl(\prod_{i=1}^n \hat{\mathbf b}_i\biggr).
\]
Since $\mathsf{con}(\hat{\bar{\mathbf c}})=\mathsf{con}(\hat{\mathbf c})$ for any $c\in \mathscr V^\prime$ and the identity $\mathbf u\approx \mathbf v$ is reduced, the identity $\mathbf w\approx \mathbf v$ is $(r-r^\prime)$-invertible with
\[
r^\prime=\sum_{i\in\{j\mid \xi_j\ne\eta_j\}}|\hat{\mathbf x}_{1\xi_i}^{(i)}|\cdot|\hat{\mathbf x}_{2\xi_i}^{(i)}|>0.
\]
Clearly, the identity $\mathbf w\approx \mathbf v$ is reduced.
So, we can apply the induction assumption,  yielding that $\mathbb M(\mathscr W_n)\{\mathbf w\approx\mathbf v\}=\mathbb M(\mathscr W_n)\{\Psi\}$ for some $\Psi\subseteq \mathsf{Id}(\upsilon_{\mathscr W_n})$.
Further, $\mathbb M(\mathscr W_n)\{\mathbf w_\xi\approx\mathbf w_\eta\}$ satisfies the identities
\begin{align*}
\mathbf u&=\mathbf a\cdot\mathbf p\mathbf h\mathbf q\mathbf r\cdot\mathbf b\\
&\approx \mathbf a\cdot\mathbf p\mathbf h\mathbf q\tilde{\mathbf r}\cdot\mathbf b&&\text{by Lemma~\ref{lem:swapping}}\\
&\approx \mathbf a\cdot\mathbf p\tilde{\mathbf h}\mathbf q\tilde{\mathbf r}\cdot\mathbf b&&\text{by $\mathbf w_\xi\approx\mathbf w_\eta$}\\
&\approx \mathbf a\cdot\mathbf p\tilde{\mathbf h}\mathbf q\mathbf r\cdot\mathbf b&&\text{by Lemma~\ref{lem:swapping}},\\
&=\mathbf w,
\end{align*}
where
\[
\tilde{\mathbf r} :=(\check{\mathbf b})_{\mathsf{con}(\hat{\mathbf b})}\hat{\mathbf b}\hat{\mathbf  y}_0\,\biggl(\prod_{i=1}^n \hat{\mathbf x}_1^{(i)}\,\hat{\mathbf z}_i \hat{\mathbf a}_i\hat{\mathbf z}_i^{\prime}\hat{\mathbf b}_i\hat{\mathbf z}_i^{\prime\prime} \,\hat{\mathbf x}_2^{(i)}\, \hat{\mathbf  y}_i\biggr)\,\hat{\mathbf a}(\check{\mathbf a})_{\mathsf{con}(\hat{\mathbf a})}.
\]
Since $\mathbf w_\xi\approx\mathbf w_\eta$ is a consequence of $\mathbf u\approx \mathbf v$, this implies that
\[
\mathbb M(\mathscr W_n)\{\mathbf u\approx\mathbf v\}= \mathbb M(\mathscr W_n)\{\mathbf u\approx\mathbf w\approx\mathbf v\}=\mathbb M(\mathscr W_n)\{\mathbf w_\xi\approx\mathbf w_\eta,\,\Psi\}.
\]

\subsection*{The mapping $\Phi$ is an anti-isomorphism}

Let $\pi,\rho \in  \mathfrak{Eq}(\mathscr W_n)$.
If $\pi \subseteq \rho$, then the inclusion $\mathbb M(\mathscr W_n)\{\mathsf{Id}(\rho)\} \subseteq \mathbb M(\mathscr W_n)\{\mathsf{Id}(\pi)\}$ holds, so that $\Phi(\rho) \subseteq \Phi(\pi)$.
Conversely, assume the inclusion $\Phi(\rho) \subseteq \Phi(\pi)$, so that $\mathbb M(\mathscr W_n)\{\mathsf{Id}(\rho)\} \subseteq \mathbb M(\mathscr W_n)\{\mathsf{Id}(\pi)\}$.
Then for any $(\mathbf u,\mathbf v) \in \pi$, the identity $\mathbf u \approx \mathbf v$ is satisfied by $\mathbb M(\mathscr W_n)\{\mathsf{Id}(\rho)\}$, whence $(\mathbf u,\mathbf v) \in \rho$ by Corollary~\ref{cor:FIC(M(W_n)Id(pi))-class}.
Therefore $\pi \subseteq \rho$.\qed

\section{Some open problems} 
\label{sec:problems}

\subsection{Monoids of order at least six}

In the present article, we show that the 6-element monoids $A_2^1$ and $B_2^1$ generate finitely universal varieties.
We do not know any of 6-element monoids distinct from $A_2^1$ and $B_2^1$ generating varieties with this property.
Thus, the following question is relevant.

\begin{question}
\label{que:six-el}
Is there a 6-element monoid distinct from $A_2^1$ and $B_2^1$ generating a finitely universal variety?
\end{question}

As we have mentioned above, the varieties $\mathbb B_2^1$ and $\mathbb A_2^1$ are non-finitely based. 
The following question is still open.

\begin{question}
\label{que:fb-fu}
What is the least order of a finitely based monoid that generates a finitely universal variety?
\end{question}

It is shown in \cite{Lee-Li-11} that every monoid of order six distinct from $A_2^1$ and $B_2^1$ generates a finitely based variety.
In view of this result, the affirmative answer to Question~\ref{que:six-el} provides a solution to Question \ref{que:fb-fu}.

\subsection{Lattice universal varieties}

Here we remind an open question from~\cite{Gusev-Lee-20} and~~\cite{Gusev-Lee-Vernikov-22}.
It follows from~\cite{Pudlak-Tuma-80} that a variety $\mathbb V$ is finitely universal if and only if for all sufficiently large $n \ge 1$, the lattice $\mathfrak{Eq}(\{1,2,\dots,n\})$ is anti-isomorphic to some sublattice of $\mathfrak L(\mathbb V)$.
In the present article, finitely universal varieties $\mathbb V$ of monoids are exhibited with the stronger property that for all sufficiently large $n \ge 1$, the lattice $\mathfrak{Eq}(\{1,2,\ldots,n\})$ is anti-isomorphic to some subinterval of $\mathfrak L(\mathbb V)$.
A yet even stronger property that a variety $\mathbb V$ can satisfy is when the lattice $\mathfrak{Eq}(\{1,2,3,\dots\})$ is anti-isomorphic to some subinterval of $\mathfrak L(\mathbb V)$; following~\cite{Shevrin-Vernikov-Volkov-09}, such a variety is said to be \textit{lattice universal}.
Lattice universal varieties of semigroups have been found in~\cite{Burris-Nelson-71a} and~\cite{Jezek-76}; it is natural to question if a variety of monoids can also satisfy this property.

\begin{question}[{\!\cite[Question~6.5]{Gusev-Lee-20}}; see also {\cite[Question~4.11b)]{Gusev-Lee-Vernikov-22}}]
\label{que:lu}
Is there a variety of monoids that is lattice universal?
\end{question}

Notice that, for locally finite varieties the answer to Question~\ref{que:lu} is negative.
This immediately follows from the following three folkloric facts: the subvariety lattice of an arbitrary locally finite variety is algebraic; the lattice $\mathfrak{Eq}(\{1,2,3,\dots\})$ is not coalgebraic; an interval of an algebraic lattice is again algebraic.

\subsection*{Acknowledgments.} The author thanks Edmond W.H. Lee for several comments and suggestions for improving the manuscript and many discussions.

\end{document}